 \documentclass[a4paper,12pt]{article}
    \usepackage[top=2.5cm,bottom=2.5cm,left=2.5cm,right=2.5cm]{geometry}
    \usepackage{cite, amsmath, amssymb}
    \usepackage[margin=1cm,%
                font=small,%
                format=hang,%
                labelsep=period,%
                labelfont=bf]{caption}
\usepackage[latin1]{inputenc}
\usepackage{amsmath}
\usepackage{amsfonts}
\usepackage{amssymb}
\usepackage{cite}
\usepackage{amsthm}
\usepackage{makeidx}
\usepackage{graphicx}
\usepackage{mathrsfs,xcolor}
\usepackage{enumerate}
\usepackage{blkarray}
\usepackage{bm}
\usepackage{setspace}
\usepackage{kbordermatrix}
\usepackage{float}
\kbalignrighttrue
\renewcommand{\kbldelim}{(}
\renewcommand{\kbrdelim}{)}
\renewcommand{\kbrowstyle}{\displaystyle}
\renewcommand{\kbcolstyle}{\displaystyle}

\usepackage{authblk}

\numberwithin{table}{section}
\numberwithin{equation}{section}

\theoremstyle{plain}
\newtheorem{theorem}{Theorem}[section]
\newtheorem{proposition}[theorem]{Proposition}
\newtheorem{definition}[theorem]{Definition}
\newtheorem{lemma}[theorem]{Lemma}
\newtheorem{example}[theorem]{Example}
\newtheorem{runningexample}[theorem]{Running Example}
\newtheorem{corollary}[theorem]{Corollary}
\newtheorem{remark}[theorem]{Remark}

\usepackage{authblk}
\author[1]{ \textbf{Bryan S. Hernandez}}
 \author[2,3,4,5]{\textbf{Eduardo R. Mendoza}}
 \author[1,*]{\textbf{Aurelio A. de los Reyes V}} 

\affil[1]{\small \textit{Institute of Mathematics, University of the Philippines Diliman, Quezon City 1101, Philippines }}
\affil[2]{\small \textit{Institute of Mathematical Sciences and Physics, University of the Philippines,  Los Ba\~{n}os, Laguna 4031, Philippines}}
\affil[3]{\small \textit{Mathematics and Statistics Department, De La Salle University, Manila  0922, Philippines}}
\affil[4]{\small \textit{Max Planck Institute of Biochemistry, Martinsried, Munich, Germany}}
\affil[5]{\small \textit{LMU Faculty of Physics, Geschwister -Scholl- Platz 1, 80539 Munich, Germany}}
\affil[*]{Corresponding author: \texttt{adlreyes@math.upd.edu.ph}}

\title{\vspace{3.5cm}\textbf{Fundamental Decompositions and Multistationarity of Power-Law Kinetic Systems}}

\date{\normalsize (Received October 15, 2019)}

\begin{document}
\maketitle
\thispagestyle{empty}
\begin{abstract}
The fundamental decomposition of a chemical reaction network (also called its ``$\mathscr{F}$-decomposition'') is the set of subnetworks generated by the partition of its set of reactions into the ``fundamental classes'' introduced by Ji and Feinberg in 2011 as the basis of their ``higher deficiency algorithm'' for mass action systems. The first part of this paper studies the properties of the $\mathscr{F}$-decomposition, in particular, its independence (i.e., the network's stoichiometric subspace is the direct sum of the subnetworks' stoichiometric subspaces) and its incidence-independence (i.e., the image of the network's incidence map is the direct sum of the incidence maps' images of the subnetworks). We derive necessary and sufficient conditions for these properties and identify network classes where the $\mathscr{F}$-decomposition coincides with other known decompositions. The second part of the paper applies the above-mentioned results to improve the Multistationarity Algorithm for power-law kinetic systems (MSA), a general computational approach that we introduced in previous work. We show that for systems with non-reactant determined interactions but with an independent $\mathscr{F}$-decomposition, the transformation to a dynamically equivalent system with reactant-determined interactions -- required in the original MSA -- is not necessary. We illustrate this improvement with the subnetwork of Schmitz's carbon cycle model recently analyzed by Fortun et al.
\end{abstract}
\baselineskip=0.30in

\section{Introduction}
\label{intro}
The fundamental decomposition (also called ``$\mathscr{F}$-decomposition'') of a chemical reaction network (CRN) is the set of subnetworks generated by the partition of its set of reactions into the ``fundamental classes'' introduced by Ji and Feinberg in 2011 as the basis of their Higher Deficiency Algorithm (HDA) for mass action systems. For a CRN with only irreversible reactions and stoichiometric matrix $N$, the characteristic functions $\omega _r$ and $\omega _{r'}$ of reactions $r$ and $r'$ are in the same non-zero fundamental class $C_i$ if they are non-zero and pairwise dependent in the factor space ${\mathbb{R}^\mathscr{R}/({\text{Ker\ }} N)^{\perp}}$. Any reaction with $\omega _r= 0$ in the factor space is assigned to the zero fundamental class $C_0$. In the general case, a reversible reaction pair is also assigned to the same fundamental class. 

The first part of this paper studies the properties of the $\mathscr{F}$-decomposition, in particular, its
independence (i.e., the network's stoichiometric subspace is the direct sum of the subnetworks' stoichiometric subspaces) and its incidence-independence (i.e., the image of the network's incidence map is the direct sum of the incidence maps' images of the subnetworks). For deficiency zero networks, these two properties coincide. M. Feinberg established the essential relationship between independent decompositions and the set of positive equilibria of a network (recalled as Theorem \ref{feinberg:decom:thm}) in 1987 \cite{feinberg12}. A corresponding relationship between incidence-independent, weakly reversible decompositions and complex-balanced equilibria of a weakly reversible network (recalled as Theorem \ref{decomposition:thm:2}) was recently documented by Farinas et al. \cite{FML2018}. We derive necessary and sufficient conditions for the independence and incidence of $\mathscr{F}$-decompositions, and identify network classes where the $\mathscr{F}$-decomposition coincides with other known decompositions. 

In our previous work \cite{hernandez}, we showed that the HDA of Ji and Feinberg can be extended to any power-law kinetic system which has reactant-determined interactions (we denote the set by PL-RDK), i.e., the reactions branching from the same reactant complex have identical kinetic order vectors (or ``interactions''). By combining this extension with a method (called CF-RM) to transform a power-law kinetic system with non-reactant-determined interactions (denoted by PL-NDK) to a dynamically equivalent PL-RDK system, we developed a computational approach called the ``Multistationarity Algorithm'' (MSA) to determine multistationarity in a stoichiometric class for any power-law kinetic system.

The second part of the paper applies the results of the first part to improve the MSA for power-law kinetic systems. We show that for PL-NDK systems with an independent $\mathscr{F}$-decomposition, the extended HDA, can be directly applied just as in the case of a PL-RDK system. This is because the new transformation, denoted by CF-RI$_+$, preserves reversibility/ irreversibility of reactions, leading to identical multistationarity computations for such PL-NDK systems and their PL-RDK transforms via CF-RI$_+$. The subnetwork of Schmitz's global carbon cycle model, a PL-NDK system studied by Fortun et al. \cite{fortun2}, is used as a running example to illustrate this improvement of the MSA.

The paper is organized as follows: Section \ref{prelim} collects the fundamentals of chemical reaction networks and kinetic systems required for the later sections, including relevant results from decomposition theory.  After introducing the $\mathscr{F}$-decomposition and related constructs, Section \ref{orientations:decompositions} derives its basic properties, including bounds for the number of subnetworks and related necessary conditions for independence and incidence-independence. Network classes with independent or incidence-independent $\mathscr{F}$-decompositions are identified. In Section \ref{decomposition:types:network:prop}, Ji and Feinberg's characterization of $\mathscr{F}$-decomposition subnetworks is used to develop a classification of $\mathscr{F}$-decompositions. Bounds for the deficiency and other properties are then obtained for the three $\mathscr{F}$-decomposition types. The reaction reversibility/ irreversibility preserving transformation CF-RI$_+$ is introduced in Section \ref{sect:cfri:transform:method}. This forms the basis for the improvement of the Multistationarity Algorithm (MSA) derived in Section \ref{f:decomposition:cfri:trans}.  Conclusions and an outlook constitute Section \ref{sec:conclusion}. Tables of acronyms and frequently used symbols are provided in Appendix \ref{nomenclature:appendix}.

\section{Fundamentals of Chemical Reaction Networks and Kinetic Systems}
\label{prelim}
\indent In this section, we recall some fundamental notions about chemical reaction networks and chemical kinetic systems. These concepts are provided in \cite{arc_jose,feinberg}. Moreover, we present some important preliminaries on the decomposition theory which was introduced by Feinberg in \cite{feinberg12}.

\subsection{Fundamentals of Chemical Reaction Networks}

\begin{definition}
A {\bf chemical reaction network} (CRN) $\mathscr{N}$ is a triple $\left(\mathscr{S},\mathscr{C},\mathscr{R}\right)$ of nonempty finite sets where $\mathscr{S}$, $\mathscr{C}$, and $\mathscr{R}$ are the sets of $m$ species, $n$ complexes, and $r$ reactions, respectively, such that
$\left( {{C_i},{C_i}} \right) \notin \mathscr{R}$ for each $C_i \in \mathscr{C}$; and
for each $C_i \in \mathscr{C}$, there exists $C_j \in \mathscr{C}$ such that $\left( {{C_i},{C_j}} \right) \in \mathscr{R}$ or $\left( {{C_j},{C_i}} \right) \in \mathscr{R}$.
\end{definition}

\begin{definition}
The {\bf molecularity matrix}, denoted by $Y$, is an $m\times n$ matrix such that $Y_{ij}$ is the stoichiometric coefficient of species $X_i$ in complex $C_j$.
The {\bf incidence matrix}, denoted by $I_a$, is an $n\times r$ matrix such that 
$${\left( {{I_a}} \right)_{ij}} = \left\{ \begin{array}{rl}
 - 1&{\rm{ if \ }}{C_i}{\rm{ \ is \ in \ the\ reactant \ complex \ of \ reaction \ }}{R_j},\\
 1&{\rm{  if \ }}{C_i}{\rm{ \ is \ in \ the\ product \ complex \ of \ reaction \ }}{R_j},\\
0&{\rm{    otherwise}}.
\end{array} \right.$$
The {\bf stoichiometric matrix}, denoted by $N$, is the $m\times r$ matrix given by 
$N=YI_a$.
\end{definition}
Let $\mathscr{I}=\mathscr{S}, \mathscr{C}$ or $\mathscr{R}$. We denote the standard basis for $\mathbb{R}^\mathscr{I}$ by $\left\lbrace \omega_i \in \mathbb{R}^\mathscr{I} \mid i \in \mathscr{I} \right\rbrace$.

\begin{definition}
Let $\mathscr{N}=(\mathscr{S,C,R})$ be a CRN. The {\bf{incidence map}} $I_a : \mathbb{R}^\mathscr{R} \rightarrow \mathbb{R}^\mathscr{C}$ is the linear map such that for each reaction $r:C_i \rightarrow C_j \in \mathscr{R}$, the basis vector $\omega_r$ to the vector $\omega_{C_j}-\omega_{C_i} \in \mathscr{C}$.
\end{definition} 

\begin{definition}
The {\bf reaction vectors} for a given reaction network $\left(\mathscr{S},\mathscr{C},\mathscr{R}\right)$ are the elements of the set $\left\{{C_j} - {C_i} \in \mathbb{R}^\mathscr{S}|\left( {{C_i},{C_j}} \right) \in \mathscr{R}\right\}.$
\end{definition}

\begin{definition}
The {\bf stoichiometric subspace} of a reaction network $\left(\mathscr{S},\mathscr{C},\mathscr{R}\right)$, denoted by $S$, is the linear subspace of $\mathbb{R}^\mathscr{S}$ given by $S = span\left\{ {{C_j} - {C_i} \in \mathbb{R}^\mathscr{S}|\left( {{C_i},{C_j}} \right) \in \mathscr{R}} \right\}.$ The {\bf rank} of the network, denoted by $s$, is given by $s=\dim S$. The set $\left( {x + S} \right) \cap \mathbb{R}_{ \ge 0}^\mathscr{S}$ is said to be a {\bf stoichiometric compatibility class} of $x \in \mathbb{R}_{ \ge 0}^\mathscr{S}$.
\end{definition}

\begin{definition}
Two vectors $x, x^{*} \in {\mathbb{R}^\mathscr{S}}$ are {\bf stoichiometrically compatible} if $x-x^{*}$ is an element of the stoichiometric subspace $S$.
\end{definition}

We can view complexes as vertices and reactions as edges. With this, CRNs can be seen as graphs. At this point, if we are talking about geometric properties, {\bf vertices} are complexes and {\bf edges} are reactions. If there is a path between two vertices $C_i$ and $C_j$, then they are said to be {\bf connected}. If there is a directed path from vertex $C_i$ to vertex $C_j$ and vice versa, then they are said to be {\bf strongly connected}. If any two vertices of a subgraph are {\bf (strongly) connected}, then the subgraph is said to be a {\bf (strongly) connected component}. The (strong) connected components are precisely the {\bf (strong) linkage classes} of a CRN. The maximal strongly connected subgraphs where there are no edges from a complex in the subgraph to a complex outside the subgraph is said to be the {\bf terminal strong linkage classes}.
We denote the number of linkage classes and the number of strong linkage classes by $l$ and $sl$, respectively.
A CRN is said to be {\bf weakly reversible} if $sl=l$.

\begin{definition}
For a CRN, the {\bf deficiency} is given by $\delta=n-l-s$ where $n$ is the number of complexes, $l$ is the number of linkage classes, and $s$ is the dimension of the stoichiometric subspace $S$.
\end{definition}

\subsection{Fundamentals of Chemical Kinetic Systems}

\begin{definition}
A {\bf kinetics} $K$ for a reaction network $\left(\mathscr{S},\mathscr{C},\mathscr{R}\right)$ is an assignment to each reaction $r: y \to y' \in \mathscr{R}$ of a rate function ${K_r}:{\Omega _K} \to {\mathbb{R}_{ \ge 0}}$ such that $\mathbb{R}_{ > 0}^\mathscr{S} \subseteq {\Omega _K} \subseteq \mathbb{R}_{ \ge 0}^\mathscr{S}$, $c \wedge d \in {\Omega _K}$ if $c,d \in {\Omega _K}$, and ${K_r}\left( c \right) \ge 0$ for each $c \in {\Omega _K}$.
Furthermore, it satisfies the positivity property: supp $y$ $\subset$ supp $c$ if and only if $K_r(c)>0$.
The system $\left(\mathscr{S},\mathscr{C},\mathscr{R},K\right)$ is called a {\bf chemical kinetic system}.
\end{definition}

\begin{definition}
The {\bf species formation rate function} (SFRF) of a chemical kinetic system is given by $f\left( x \right) = NK(x)= \displaystyle \sum\limits_{{C_i} \to {C_j} \in \mathscr{R}} {{K_{{C_i} \to {C_j}}}\left( x \right)\left( {{C_j} - {C_i}} \right)}.$
\end{definition}
The ordinary differential equation (ODE) or dynamical system of a chemical kinetics system is $\dfrac{{dx}}{{dt}} = f\left( x \right)$. An {\bf equilibrium} or {\bf steady state} is a zero of $f$.

\begin{definition}
The {\bf set of positive equilibria} of a chemical kinetic system $\left(\mathscr{S},\mathscr{C},\mathscr{R},K\right)$ is given by ${E_ + }\left(\mathscr{S},\mathscr{C},\mathscr{R},K\right)= \left\{ {x \in \mathbb{R}^\mathscr{S}_{>0}|f\left( x \right) = 0} \right\}.$
\end{definition}

A CRN is said to admit {\bf multiple equilibria} if there exist positive rate constants such that the ODE system admits more than one stoichiometrically compatible equilibria.

\begin{definition}
A kinetics $K$ is {\bf complex factorizable} if, for $K(x) = k I_K(x)$,  the interaction map $I_K : \mathbb{R}^\mathscr{S} \to \mathbb{R}^\mathscr{R}$
factorizes via the space of complexes $\mathbb{R}^\mathscr{C} : I_K = I_k \circ \psi _K$ with  $\psi _K: \mathbb{R}^\mathscr{S} \to \mathbb{R}^\mathscr{C}$ as factor map and
$I_k = diag(k) \circ \rho'$ with $\rho' : \mathbb{R}^\mathscr{C} \to \mathbb{R}^\mathscr{R}$ assigning the value at a reactant complex to all its reactions. 

\end{definition}

\begin{definition}
A kinetics $K$ is a {\bf power-law kinetics} (PLK) if 
${K_i}\left( x \right) = {k_i}{{x^{{F_{i}}}}} $ for $i =1,...,r$ where ${k_i} \in {\mathbb{R}_{ > 0}}$ and ${F_{ij}} \in {\mathbb{R}}$. The power-law kinetics is defined by an $r \times m$ matrix $F$, called the {\bf kinetic order matrix} and a vector $k \in \mathbb{R}^\mathscr{R}$, called the {\bf rate vector}.
\end{definition}
If the kinetic order matrix is the transpose of the molecularity matrix, then the system becomes the well-known {\bf mass action kinetics (MAK)}.

\begin{definition}
A PLK system has {\bf reactant-determined kinetics} (of type PL-RDK) if for any two reactions $i, j$ with identical reactant complexes, the corresponding rows of kinetic orders in $F$ are identical, i.e., ${f_{ik}} = {f_{jk}}$ for $k = 1,2,...,m$. A PLK system has {\bf non-reactant-determined kinetics} (of type PL-NDK) if there exist two reactions with the same reactant complexes whose corresponding rows in $F$ are not identical.
\end{definition}

\subsection{Review of Decomposition Theory}
\label{sect:decomposition}

In this subsection, we recall some definitions and earlier results from the decomposition theory of chemical reaction networks. 

\begin{definition}
A {\bf decomposition} of $\mathscr{N}$ is a set of subnetworks $\{\mathscr{N}_1, \mathscr{N}_2,...,\mathscr{N}_k\}$ of $\mathscr{N}$ induced by a partition $\{\mathscr{R}_1, \mathscr{R}_2,...,\mathscr{R}_k\}$ of its reaction set $\mathscr{R}$. 
\end{definition}

We denote a decomposition with 
$\mathscr{N} = \mathscr{N}_1 \cup \mathscr{N}_2 \cup ... \cup \mathscr{N}_k$
since $\mathscr{N}$ is a union of the subnetworks in the sense of \cite{GHMS2018}. It also follows immediately that, for the corresponding stoichiometric subspaces, 
${S} = {S}_1 + {S}_2 + ... + {S}_k$. It is also useful to consider refinements and coarsenings of decompositions.

\begin{definition}
A network decomposition $\mathscr{N} = \mathscr{N}_1 \cup \mathscr{N}_2 \cup ... \cup \mathscr{N}_k$  is a {\bf refinement} of
$\mathscr{N} = {\mathscr{N}'}_1 \cup {\mathscr{N}'}_2 \cup ... \cup {\mathscr{N}'}_{k'}$ 
(and the latter a coarsening of the former) if it is induced by a refinement  
$\{\mathscr{R}_1, \mathscr{R}_2,...,\mathscr{R}_k\}$
of $\{{\mathscr{R}'}_1 \cup {\mathscr{R}'}_2 \cup ... \cup {\mathscr{R}'}_{k'}\}$, i.e., 
each ${\mathscr{R}}_i$ is contained in an ${\mathscr{R}'}_j$. 
\end{definition}

In \cite{feinberg12}, Feinberg introduced the important concept of independent decomposition.

\begin{definition}
A network decomposition $\mathscr{N} = \mathscr{N}_1 \cup \mathscr{N}_2 \cup ... \cup \mathscr{N}_k$  is {\bf independent} if its stoichiometric subspace is a direct sum of the subnetwork stoichiometric subspaces.
\end{definition}

In Lemma 1 of \cite{fortun2}, it was shown that for an independent decomposition, $$\delta \le \delta_1 +\delta_2 ... +\delta_k.$$

\begin{definition}
A decomposition $\mathscr{N} = \mathscr{N}_1 \cup \mathscr{N}_2 \cup ... \cup \mathscr{N}_k$ with $\mathscr{N}_i = (\mathscr{S}_i,\mathscr{C}_i,\mathscr{R}_i)$ is a {\bf $\mathscr{C}$-decomposition} if for each pair of distinct $i$ and $j$, $\mathscr{C}_i$ and $\mathscr{C}_j$ are disjoint.
\end{definition}

\begin{example}
Linkage classes form the primary example of a decomposition of a CRN. They are special in the sense that the reaction set partition inducing them is also a complex set partition. In \cite{FML2018}, the term ``$\mathscr{C}$-decomposition'' was introduced for such a decomposition and it was shown that any $\mathscr{C}$-decomposition is a coarsening of the linkage class decomposition.
\end{example}

For any decomposition, it also holds that 
$\text{Im } I_a = \text{Im } I_{a,1} + ... + \text{Im } I_{a,k}$, where $$\text{Im } I_{a,i} = I_a(\mathbb{R}^{{\mathscr{R}}_i}).$$

\begin{definition}
A decomposition is {\bf incidence-independent} if ${\rm{Im \ }} I_a$ is a direct sum of the ${\rm{Im \ }} I_{a,i}$. It is {\bf bi-independent} if it is both independent and incidence-independent.
\end{definition}

An equivalent formulation of showing incidence-independent is to satisfy $n - l = \sum {\left( {{n_i} - {l_i}} \right)}$, where $n_i$ is the number of complexes and $l_i$ is the number of linkage classes, in each subnetwork $i$.

In \cite{FML2018}, it was shown that for any incidence-independent decomposition, $\delta \ge \delta_1 +\delta_2 ... +\delta_k$ and that $\mathscr{C}$-decompositions form a subset of incidence-independent decompositions. Hence, independent linkage classes form the primary example of a bi-independent decomposition.

The following proposition is easily verified:
\begin{proposition}
A decomposition $\mathscr{N} = \mathscr{N}_1 \cup \mathscr{N}_2 \cup ... \cup \mathscr{N}_k$ is independent or incidence-independent and $\sum {{\delta _i} = \delta }$ iff $\mathscr{N} = \mathscr{N}_1 \cup \mathscr{N}_2 \cup ... \cup \mathscr{N}_k$ is bi-independent.
\end{proposition}

In particular, for a zero deficiency network, independence and incidence-independence are equivalent. 

Feinberg established the following basic relation between an independent decomposition and the set of positive equilibria of a kinetics on the network:

\begin{theorem} (Feinberg Decomposition Theorem \cite{feinberg12})
\label{feinberg:decom:thm}
Let $P(\mathscr{R})=\{\mathscr{R}_1, \mathscr{R}_2,...,\mathscr{R}_k\}$ be a partition of a CRN $\mathscr{N}$ and let $K$
be a kinetics on $\mathscr{N}$. If $\mathscr{N} = \mathscr{N}_1 \cup \mathscr{N}_2 \cup ... \cup \mathscr{N}_k$ is the network decomposition of $P(\mathscr{R})$ and ${E_ + }\left(\mathscr{N}_i,{K}_i\right)= \left\{ {x \in \mathbb{R}^\mathscr{S}_{>0}|N_iK_i(x) = 0} \right\}$ then
\[{E_ + }\left(\mathscr{N}_1,K_1\right) \cap {E_ + }\left(\mathscr{N}_2,K_2\right) \cap ... \cap {E_ + }\left(\mathscr{N}_k,K_k\right) \subseteq  {E_ + }\left(\mathscr{N},K\right).\]
If the network decomposition is independent, then equality holds.
\end{theorem}

The analogue of Feinberg's 1987 result for incidence-independent decompositions and complex-balanced equilibria is shown in \cite{FML2018}:

\begin{theorem}
\label{decomposition:thm:2}
Let $\mathscr{N}$ be a network, $K$ any kinetics and $\mathscr{N} = \mathscr{N}_1 \cup \mathscr{N}_2 \cup ... \cup \mathscr{N}_k$ an incidence-independent decomposition of weakly reversible subnetworks. Then $\mathscr{N}$ is weakly reversible and 
\begin{itemize}
\item[i.] ${Z_ + }\left( {\mathscr{N} ,K} \right) =  \cap {Z_ + }\left( {{\mathscr{N} _i},K} \right)$ for each subnetwork ${\mathscr{N} _i}$.
\item[ii.] If ${Z_ + }\left( {\mathscr{N},K} \right) \ne \emptyset $ then ${Z_ + }\left( {{\mathscr{N}_i},K} \right) \ne \emptyset $ for each subnetwork ${\mathscr{N} _i}$.
\item[iii.] If the decomposition is a $\mathscr{C}$-decomposition and $K$ a complex factorizable kinetics then  ${Z_ + }\left( {{\mathscr{N}_i},K} \right) \ne \emptyset $ for each subnetwork ${\mathscr{N} _i}$ implies that ${Z_ + }\left( {{\mathscr{N}},K} \right) \ne \emptyset $.
\end{itemize}
\end{theorem}

\begin{definition}
The network decomposition $\mathscr{N}=\mathscr{N'} \cup \mathscr{N''}$ is said to be {\bf trivial} if $\mathscr{N'}$ is a subnetwork whose stoichiometric subspace coincides with that of $\mathscr{N}$.
\end{definition}





\section{Orientations, and $\mathscr{O}$-, $\mathscr{P}$-, and $\mathscr{F}$-decompositions of CRNs}
\label{orientations:decompositions}
In this section, we review the concepts and properties underlying HDA and its extension to PL-RDK systems in the context of decomposition theory.

\subsection{Review of Orientations, $\mathscr{O}$-, $\mathscr{P}$-, and $\mathscr{F}$-decompositions}
The main references for this subsection are \cite{hernandez} and \cite{ji}.

\begin{definition} A subset $\mathscr{O}$ of $\mathscr{R}$ is said to be an {\bf orientation} if for every reaction $y \to y' \in \mathscr{R}$, either $y \to y' \in \mathscr{O}$ or $y' \to y \in \mathscr{O}$, but not both.
\end{definition}

For an orientation $\mathscr{O}$, we define a linear map ${L_\mathscr{O}}:{\mathbb {R}^\mathscr{O}} \to S$ such that
\[{L_\mathscr{O}}(\alpha)= \sum\limits_{y \to y' \in \mathscr{O}} {{\alpha _{y \to y'}}\left( {y' - y} \right)}.\]

Let ${{r}_{\text{irr}}}$ and ${{r}_{\text{rev}}}$ be the number of irreversible reactions and reversible reaction pairs in $\mathscr{N}$, respectively. Clearly, $r={{r}_{\text{irr}}}+2{{r}_{\text{rev}}}$. In the succeeding disscussion, we will be using the notation $\mathscr{N}_\mathscr{O}$ to denote the subnetwork of $\mathscr{N}$ with respect to the orientation $\mathscr{O}$.
The following proposition collects some basic properties of orientations:

\begin{proposition}
Let $\mathscr{N}$ be a CRN and $\mathscr{O}_\mathscr{N}$ be the set of orientations of $\mathscr{N}$. Then
\begin{itemize}
\item[i.] $\left| {\mathscr{O}_\mathscr{N}} \right| = {2^{r''}}$ where $r''=r_{\rm{rev}}$, and
\item[ii.] the map $P:{\mathscr{O}_\mathscr{N}} \to \{\mathscr{N}_\mathscr{O}\}$ is bijective. Hence, there are ${2^{r''}}$
$\mathscr{O}$-subnetworks in $\mathscr{N}$.
\end{itemize}
\end{proposition}

Each orientation $\mathscr{O}$ defines a partition of $\mathscr{N}$ into $\mathscr{O}$ and its complement $\mathscr{O}'$, which generates the following decomposition:

\begin{definition}
For an orientation $\mathscr{O}$ on $\mathscr{N}$, the {\bf $\mathscr{O}$-decomposition} of $\mathscr{N}$ consists of the subnetworks $\mathscr{N}_\mathscr{O}$ and $\mathscr{N}_{\mathscr{O}'}$, i.e., $\mathscr{N}=\mathscr{N}_\mathscr{O} \cup \mathscr{N}_\mathscr{O'}$.
\end{definition}

A basic property of an $\mathscr{O}$-decomposition is the following proposition.
\begin{proposition}
Let $\mathscr{N}=\left(\mathscr{S},\mathscr{C},\mathscr{R}\right)$ be a CRN and $\mathscr{O}$ be an orientation. Then the $\mathscr{O}$-decomposition is a trivial decomposition.
\label{trivial_decomposition}
\end{proposition}
\begin{proof}
Let $\mathscr{O}$ be an orientation and $\mathscr{O}'=\mathscr{R} \backslash \mathscr{O}$. Then $\mathscr{R} = \mathscr{O} \cup \mathscr{O}'$ is a partition with corresponding network decomposition $\mathscr{N} = {\mathscr{N}_\mathscr{O}} \cup {\mathscr{N}_\mathscr{O'}}$. Note that $S=S_\mathscr{O} + S_\mathscr{O'} = S_\mathscr{O}$. Hence, the stoichiometric subspace of ${\mathscr{N}_\mathscr{O}}$ is the same as that of the whole stoichiometric subspace of $\mathscr{N}$.
\end{proof}

\begin{corollary}
For any CRN $\mathscr{N}$ and orientation $\mathscr{O}$, $\delta \left( \mathscr{N} \right) = \delta \left( {{\mathscr{N}_\mathscr{O}}} \right)$.
\end{corollary}


We now review the important concept of ``equivalence classes'' from \cite{ji}.
Let $\left\{ {{v^l}} \right\}_{l = 1}^d$ be a basis for $Ker{L_\mathscr{O}}$.
If for $y \to y' \in \mathscr{O}$, $v_{y \to y'}^l =0$ for all $1 \le l \le d$ then the reaction $y \to y' $ belongs to the zeroth equivalence class $P_0$.
For $y \to y', {\overline y  \to \overline y '} \in \mathscr{O} \backslash P_0$, if there exists $\alpha \ne 0$ such that $v_{y \to y'}^l = \alpha v_{\overline y  \to \overline y '}^l$ for all $1 \le l \le d$, then the two reactions are in the same equivalence class denoted by $P_i$, $i \ne 0$.

\begin{definition}
The {\bf $\mathscr{P}$-decomposition} of the $\mathscr{O}$-subnetwork $\mathscr{N}_\mathscr{O}$ is the decomposition induced by the partition of $\mathscr{O}$ into equivalence classes. 
\end{definition}

\begin{proposition}
Let $\mathscr{O}_\mathscr{N}$ be the set of orientations of a CRN.
\begin{itemize}
\item[i.] The map $P_d:\mathscr{O}_\mathscr{N} \to \{\mathscr{P}{\text{-decompositions}}\}$ is bijective. For any two orientations $\mathscr{O}$ and $\mathscr{O}'$, $\dim {\rm Ker \ } L_\mathscr{O}=\dim {\rm Ker \ } L_\mathscr{O'}={{r}_{\rm irr }}$ + ${{r}_{\rm{rev}}} -s$. Hence, the $\mathscr{P}$-decomposition of $\mathscr{N}_\mathscr{O}$ has the same number of subnetworks as that of $\mathscr{N}_{\mathscr{O}'}$ (denoted by $w$ if $P_0 = \emptyset$ and $w+1$ if $P_0 \ne \emptyset$).
\item[ii.] For the $i^{\text th}$ subnetworks of the $\mathscr{P}$-decompositions of $\mathscr{N}_\mathscr{O}$ and $\mathscr{N}_{\mathscr{O}'}$, respectively, the stoichiometric subspaces and the incidence map images coincide, i.e., 
$S_i=S_{i}^{'}$, and ${\rm Im } \ I_{a,i}={\rm{Im}}\ I_{a,i}^{'}$. 
\end{itemize}
\end{proposition}

\begin{proof}
For (i), we note that $\dim {\rm Ker\ } L_\mathscr{O}={{r}_{\text{irr}}}$ + ${{r}_{\text{rev}}} -s$ follows immediately from the Rank-Nullity Theorem. For (ii), in the stoichiometric subspaces $S_i$ and $S_{i}^{'}$, the reaction vectors from irreversible reactions are identical, and those from reversible pairs are either identical or the negative of each other. Hence, the stoichiometric subspaces coincide. Similarly, an element of ${\rm Im } \ I_{a,i}$, say $z=\sum {{\alpha _i}\left( {{\omega _{{y_i}}}' - {\omega _{{y_i}}}} \right)}  + \sum {{\beta _i}\left( {{\omega _{{y_i}}}' - {\omega _{{y_i}}}} \right)} $, where the first sum is over identical reactions (irreversible and identical pair reaction) and the second from the converse pair reaction. Rewriting the latter as $ - \sum {{\beta _i}\left( {{\omega _{{y_i}}} - {\omega _{{y_i}}}'} \right)} $ shows that it belongs to ${\rm{Im}}\ I_{a,i}^{'}$, and vice versa.
\end{proof}

We can now introduce the central concept of ``fundamental classes'', which is the basis of the Higher Deficiency Algorithm of Ji and Feinberg. The reactions $y \to y'$ and $\overline y  \to \overline y '$ in $\mathscr{R}$ belong to the same {\bf fundamental class} if at least one of the following is satisfied \cite{ji}.
\begin{itemize}
\item[i.] $y \to y'$ and $\overline y  \to \overline y '$ are the same reaction.
\item[ii.] $y \to y'$ and $\overline y  \to \overline y '$ are reversible pair.
\item[iii.] Either $y \to y'$ or $y' \to y$, and either $\overline y  \to \overline y '$ or $\overline y'  \to \overline y $ are in the same equivalence class on $\mathscr{O}$.
\end{itemize}

We can easily see that the orientation $\mathscr{O}$ is partitioned into equivalence classes while the reaction set $\mathscr{R}$ is partitioned into fundamental classes.

\begin{definition}
The {\bf $\mathscr{F}$-decomposition} of $\mathscr{N}$ is the decomposition generated by the partition of $\mathscr{R}$ into fundamental classes.
\end{definition}

\begin{proposition}
Any $\mathscr{P}$-decomposition generates the (unique) $\mathscr{F}$-decomposition of $\mathscr{N}$. 
\end{proposition}

\begin{proof}
Note that each partition set (``equivalence class'') of any $\mathscr{P}$-decomposition is contained in a unique partition set (``fundamental class'') of the $\mathscr{F}$-decomposition. Hence, every subnetwork of any $\mathscr{P}$-decomposition is contained in a unique subnetwork of the $\mathscr{F}$-decomposition. Conversely, each $\mathscr{F}$-subnetwork contains a unique $\mathscr{P}$-subnetwork.
\end{proof}

An upper bound for the number of fundamental classes is clearly given by $(r_{\text{irr}} + r_{\text{rev}})$, which is the number of reactions in any orientation $\mathscr{O}$.  It follows that $w \le r_{\text{irr}} + r_{\text{rev}}$. This upper bound is sharp in the sense that there are CRNs for which $w = r_{\text{irr}} + r_{\text{rev}}$. The following proposition provides a lower bound for $w$:
\begin{proposition}
For any CRN, $w \ge r_{\text{irr}} + r_{\text{rev}} - s$.
\end{proposition}
\begin{proof}
$\omega _r$ and $\omega _{r'}$  are equivalent iff $\omega _r =\alpha \omega _{r'}$ in ${\mathbb{R}^\mathscr{O}/({\text{Ker\ }} L_\mathscr{O})^{\perp}}$. The latter means that their cosets are pairwise linearly dependent in the factor space. Hence, the number of equivalence classes is at least $ \text{dim } {\mathbb{R}^\mathscr{O}/({\text{Ker\ }} L_\mathscr{O})^{\perp}} = r_{\text{irr}} + r_{\text{rev}} - s$. Since the fundamental class $C_0$ is mapped to 0 in the factor space, it follows that $w \ge r_{\text{irr}} + r_{\text{rev}} - s$.
\end{proof}

\begin{corollary}
$w \ge r_{\text{irr}} + r_{\text{rev}} - (n - l)$
\end{corollary}

\begin{runningexample}
Consider the subnetwork $\mathscr{N} = \mathscr{N}_1 \cup \mathscr{N}_2$ of the Schmitz's carbon cycle model in \cite{schmitz} that shows how the movement of carbon among different pools which represent major parts of the Earth \cite{fortun2}.
We label the reactions of the subnetwork, together with kinetic orders, depicted in Figure \ref{picschmitz}.
\begin{figure}[H]
\begin{center}
\includegraphics[width=9cm,height=18cm,keepaspectratio]{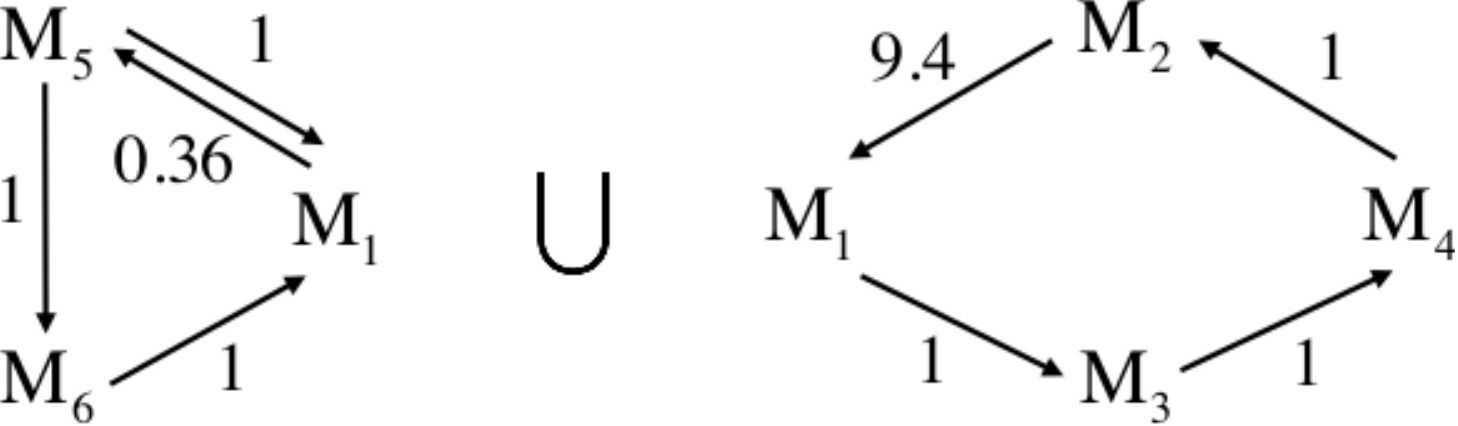}
\caption{A subnetwork of the Schmitz's carbon cycle model \cite{fortun2,schmitz}.}
\label{picschmitz}
\end{center}
\end{figure}
The following are the reactions of the subnetwork.
\[ \begin{array}{lll}
R_1: M_1 \to M_5 & \ \ \ &  R_5: M_1  \to M_3\\
R_2: M_5  \to  M_1  &  \ \ \ & R_6: M_3 \to M_4 \\
R_3: M_5  \to M_6 &  \ \ \ & R_{7}: M_4 \to M_2\\
R_4: M_6  \to M_1  &  \ \ \ & R_{8}: M_2 \to M_1  \\
\end{array}\]
We choose the orientation $\mathscr{O}=\{R_1,R_3,R_4,R_5,R_6,R_7,R_8\}$. Hence, a basis for $Ker{L_\mathscr{O}}$ is
\[\kbordermatrix{%
 ~ & v_1 & v_2 \cr
R_1 & 0 & 1 \cr
R_3 & 0 & 1 \cr
R_4 & 0 & 1 \cr
R_5 & 1 & 0 \cr
R_6 & 1 & 0 \cr
R_7 & 1 & 0 \cr
R_8 & 1 & 0 \cr
}.\]
This shows that the $\mathscr{P}$-decomposition induces the following fundamental classes:
$\{R_1,R_3,R_4\}$ and $\{R_5,R_6,R_7,R_8\}$.
Hence, it follows that the $\mathscr{F}$-decomposition has these fundamental classes:
$\{R_1,R_2,R_3,R_4\}$ and $\{R_5,R_6,R_7,R_8\}$,
which are precisely the subnetworks $\mathscr{N}_1$ and $\mathscr{N}_2$ of $\mathscr{N}$, respectively.
\end{runningexample}

\subsection{Independence and Incidence-Independence of the $\mathscr{P}$- and $\mathscr{F}$-decompositions}
This subsection shows that the essential properties of any $\mathscr{P}$-decomposition of any $\mathscr{O}$-subnetwork $\mathscr{N}_\mathscr{O}$ are fully reflected in the unique $\mathscr{F}$-decomposition of $\mathscr{N}$. We record some basic properties of independent $\mathscr{F}$-decompositions as well as provide an example of a class of CRNs with independent $\mathscr{F}$-decompositions.

\begin{theorem}
\label{PifC}
Let $\mathscr{N}_\mathscr{O}$ be the subnetwork of  $\mathscr{N}$ defined by the orientation $\mathscr{O}$ being a subset of $\mathscr{R}$. Then the following holds:
\begin{itemize}
\item [i.] The $\mathscr{P}$-decomposition of $\mathscr{N}_\mathscr{O}$ is independent if and only if the $\mathscr{F}$-decomposition of $\mathscr{N}$ is independent.
\item [ii.] The $\mathscr{P}$-decomposition of $\mathscr{N}_\mathscr{O}$ is incidence-independent if and only if the $\mathscr{F}$-decomposition of $\mathscr{N}$ is incidence-independent.
\item [iii.] The $\mathscr{P}$-decomposition of $\mathscr{N}_\mathscr{O}$ is bi-independent if and only if the $\mathscr{F}$-decomposition of $\mathscr{N}$ is bi-independent.
\end{itemize}
\end{theorem}

\begin{proof}
For (i), suppose the $\mathscr{F}$-decomposition is independent. Then, $S$ is the direct sum of the stoichiometric subspaces of the subnetworks corresponding to the fundamental classes. Note, ${P_i} \subseteq {C_i}$ for all $i$. By the definition of orientation, if $A\to B,B\to A \in \mathscr{R}$ then $A\to B \in P_i$ or $B\to A \in P_i$, but not both, for some $i$. Thus, $S$ is still the direct sum of the stoichiometric subspaces of the subnetworks corresponding to the equivalence classes. Therefore, the $\mathscr{P}$-decomposition is independent. On the other hand, suppose the $\mathscr{P}$-decomposition is independent. For the remaining reactions in $\mathscr{R}\backslash \mathscr{O}$, by definition, the reversible pairs must belong to the same fundamental class. Hence, $S$ is the direct sum of the stoichiometric subspaces of the subnetworks corresponding to the fundamental classes. Therefore, the $\mathscr{F}$-decomposition is independent. 
On the other hand, for (ii), suppose the $\mathscr{F}$-decomposition is incidence-independent. By definition, the incidence matrix of the the network is the direct sum of the incidence matrices of the fundamental classes as depicted below where each $F_i$ is indexed by complexes (rows) and reactions (columns).
\[
 \begin{pmatrix} 
    F_{0} & &  & &  \\
     & F_{1} & &\text{\huge0} & \\
     & & \ddots &  & \\
     & \text{\huge0}& & \ddots & \\
     &  &   &   & F_{w} 
    \end{pmatrix}
\]
Note ${P_i} \subseteq {C_i}$ for all $i$. Also, by definition of orientation which is partitioned by equivalence classes $P_i$'s, if a reaction is reversible, one must belong to the orientation and the other must not. Hence, if we remove one of these two reactions corresponding to two columns in the incidence matrix, the dimension of the resulting matrix will not change. This implies the incidence-independence of the $\mathscr{P}$-decomposition. Conversely, suppose the $\mathscr{P}$-decomposition is incidence-independent. Then adding the reversible pair in $P_i$ for any $i$ will not change the dimension of the incidence matrix. Thus, $\mathscr{F}$-decomposition is incidence-independent.
Statement (iii) follows from (i) and (ii).
\end{proof}

\begin{runningexample}
We again consider the subnetwork $\mathscr{N}$ of the Schmitz's carbon cycle model. Since the dimension of the stoichiometric subspaces of the fundamental classes under the $\mathscr{F}$-decomposition is equal to the dimension of the stoichiometric subspaces of $\mathscr{N}$, the $\mathscr{F}$-decomposition is independent. On the other hand, $n - l = 6 - 1 = 5$ and $\sum {\left( {{n_i} - {l_i}} \right)}  = \left( {3 - 1} \right) + \left( {4 - 1} \right) = 5$, which proves the incidence-independence of the $\mathscr{F}$-decomposition. Therefore, the said decomposition is bi-independent.
\end{runningexample}

\subsection{Independent $\mathscr{F}$-decompositions}
In this section, we present a useful necessary condition for an independent $\mathscr{F}$-decomposition and two CRN classes whose $\mathscr{F}$-decompositions are always independent. In Section \ref{decomposition:types:network:prop}, independent $\mathscr{F}$-decompositions is further analyzed by using the classification introduced by H. Ji into subnetwork types.

\subsubsection{A necessary condition for independent $\mathscr{F}$-decompositions}
\label{necessarycondition:independent}
We begin with a general property of independent decompositions.

\begin{proposition}
\label{boundforw}
Let $\mathscr{N} = \mathscr{N}_1 \cup \mathscr{N}_2 \cup ... \cup \mathscr{N}_k$ be a CRN decomposition. If the decomposition is independent, then $k \le s$. Consequently, $k \le n-l$.
\end{proposition}

\begin{proof}
Since $s_i \ge 1$, we have $\sum s_i \ge k$. If $k > s$, then  $\sum s_i  > s$, i.e., the decomposition is dependent, which shows the first claim. The second follows from $0 \le \delta$ or $s \le  n - l$.
\end{proof}

\begin{corollary}
For an independent $\mathscr{F}$-decomposition, $w \le s$.  Consequently, $w \le n - l$.
\end{corollary}

\begin{proof}
For the $\mathscr{F}$-decomposition, $w = k - 1$ (if $P_0$ is non-empty) or $w = k$ (otherwise), and the claims follow.
\end{proof}

\begin{example}
In \cite{hernandez}, we presented the CRNs of a popular model of anaerobic yeast fermentation (Section 3.1) and a model of terrestrial carbon recovery (Section 3.2). For the first CRN, we have $s =  7$ and $w =  11$  , and for the second one, $s =  4$ and $w =  6$ . It follows  that the $\mathscr{F}$-decompositions of both CRNs are not independent.  
\end{example}

\subsubsection{S-system CRNs have independent $\mathscr{F}$-decompositions}

First, we show that for S-system CRNs, the $\mathscr{F}$-decomposition is a familiar construct. We recall a definition and a result from \cite{FML2018}:

\begin{definition}
Let $R_j$ and $P_j$ be the set of variables regulating the inflow and outflow reactions of the species $X_j$ (i.e., dependent variable) of an S-system, respectively. The species is called {\bf reversible} if $R_j=P_j$. Otherwise, it is called {\bf irreversible}. An S-system is called reversible (irreversible) if all its species are reversible (irreversible).
\end{definition}

Our claim is simply the following proposition.

\begin{proposition}
For any S-system embedded CRN, the $\mathscr{F}$-decomposition is the species decomposition.
\end{proposition}

\begin{proof}
We denote the inflow reaction in $\mathscr{R}_i$ with $r_i$, the outflow with $r_{-i}$, and the corresponding basis vectors with $\omega _i$ and $\omega _{-i}$, respectively.  We set $m':= m - m_{\rm rev}$, and index the irreversible species $X_1,X_2,...,X_{m'}$. Since $|\mathscr{R}| = 2m$ and $s = m$, for any orientation, $|O|= 2m - m_{\rm{rev}}$ and $\dim {\rm Ker} L_\mathscr{O} = m - m_{\text{rev}}$. 
The $m - m_{\text{rev}}$ vectors 
$\omega _i + \omega _{-i}$, $i=1,2,...,m'$ in ${\text Ker} L_\mathscr{O}$ are linearly independent, hence form a basis. On the other hand, the $m$ vectors 
$\omega _i + \omega _{-i}, \chi _j$ with $i=1,2,...,m'$, and $j = 1,2,...,m_{\text{rev}}$ and $\chi _j$ the reaction from a reversible pair included in the orientation, form a basis for ${\text Ker} ^{\perp} L_\mathscr{O}$ . 
From the $\mathscr{F}$-decomposition definition, the reactions $\omega _i$ and $\omega _{-i}$ are equivalent, $i = 1,2,...,m$. If $k \ne i$,  $\left\langle {\omega _k - \alpha \omega _i, \omega _i + \omega _{-i}} \right\rangle =-\alpha $
, so that if $\alpha$ is nonzero, then the $k$-th inflow reaction is not equivalent. Similarly, the $k$-th outflow reaction is not equivalent. Hence, the $\mathscr{F}$-equivalence classes are precisely the $\mathscr{R}_i$'s.
\end{proof}

\begin{corollary}
The $\mathscr{F}$-decomposition of the embedded CRN of an S-system is independent.
\end{corollary}
\begin{proof}
It follows from \cite{FML2018} Theorem 1 that the species decomposition is independent, which according to the previous proposition is identical with the $\mathscr{F}$-decomposition.
\end{proof}

\subsubsection{CRNs of phosphorylation/dephosphorylation systems have independent $\mathscr{F}$-decompositions}
\label{PD:independent}
In view of their ubiquitous occurrence in cellular signaling networks, phosphorylation/ dephosphorylation (PD) systems have been extensively studied in the CRNT literature. A recent review by Conradi and Shiu \cite{CODH2018} lists several classes of multi-site PD processes which have been modeled with mass action systems -- in the following proposition, we show that the CRNs of multisite processive and distributive PD processes have distinctively different but both independent $\mathscr{F}$-decompositions.

\begin{proposition}
Let $\mathscr{N}$ be the following CRN for $k$-site processive phosphorylation/ dephosphorylation:
\[\begin{array}{c}
{S_0} + K \mathbin{\lower.3ex\hbox{$\buildrel\textstyle\rightarrow\over
{\smash{\leftarrow}\vphantom{_{\vbox to.5ex{\vss}}}}$}} {S_0}K \mathbin{\lower.3ex\hbox{$\buildrel\textstyle\rightarrow\over
{\smash{\leftarrow}\vphantom{_{\vbox to.5ex{\vss}}}}$}} {S_1}K{\mkern 1mu}  \mathbin{\lower.3ex\hbox{$\buildrel\textstyle\rightarrow\over
{\smash{\leftarrow}\vphantom{_{\vbox to.5ex{\vss}}}}$}} ...{\mkern 1mu}  \to {\mkern 1mu} {S_{k - 1}}K \to {S_k} + K\\
{S_k} + F \mathbin{\lower.3ex\hbox{$\buildrel\textstyle\rightarrow\over
{\smash{\leftarrow}\vphantom{_{\vbox to.5ex{\vss}}}}$}} {S_k}F \mathbin{\lower.3ex\hbox{$\buildrel\textstyle\rightarrow\over
{\smash{\leftarrow}\vphantom{_{\vbox to.5ex{\vss}}}}$}} {\mkern 1mu} ... \mathbin{\lower.3ex\hbox{$\buildrel\textstyle\rightarrow\over
{\smash{\leftarrow}\vphantom{_{\vbox to.5ex{\vss}}}}$}} {S_2}F{\mkern 1mu}  \mathbin{\lower.3ex\hbox{$\buildrel\textstyle\rightarrow\over
{\smash{\leftarrow}\vphantom{_{\vbox to.5ex{\vss}}}}$}} {S_1}F \to {S_0} + F
\end{array}\]
then the fundamental classes generating the $\mathscr{F}$-decomposition is the full reaction set $\mathscr{R}$. Hence, the $\mathscr{F}$-decomposition consists only of $\mathscr{N}$ and is (trivially) independent.
\end{proposition}
\begin{proof}
We choose the orientation consisting of the forward reactions and obtain:
\[\begin{array}{c}
{\alpha _k}\left( {{S_k} + K - {S_{k - 1}}K} \right) + {\alpha _{k - 1}}\left( {{S_{k - 1}}K - {S_{k - 2}}K} \right) + ... + {\alpha _1}\left( {{S_1}K - {S_0}K} \right)\\
 + {\alpha _0}\left( {{S_0}K - \left( {{S_0} + K} \right)} \right) + {\beta _0}\left( {{S_0} + F - {S_1}F} \right) + {\beta _1}\left( {{S_1}F - {S_2}F} \right)\\
 + ... + {\beta _k}\left( {{S_k}F - \left( {{S_k} + F} \right)} \right) = 0.
\end{array}\]
Thus,
\[\begin{array}{c}
{S_k}\left( {{\alpha _k} - {\beta _k}} \right) + {S_{k - 1}}K\left( { - {\alpha _k} + {\alpha _{k - 1}}} \right) + ... + {S_0}K\left( { - {\alpha _1} + {\alpha _0}} \right) + {S_0}\left( { - {\alpha _0} + {\beta _0}} \right)\\
 + {S_1}F\left( { - {\beta _0} + {\beta _1}} \right) + ... + {S_k}F\left( { - {\beta _{k - 1}} + {\beta _k}} \right) + K\left( {{\alpha _k} - {\alpha _0}} \right) + F\left( {{\beta _0} - {\beta _k}} \right) = 0.
\end{array}\]
It follows that ${\alpha _0} = {\alpha _1} = ... = {\alpha _n} = {\beta _0} = {\beta _1} = ... = {\beta _n}$ and the conclusion follows.
\end{proof}

\begin{proposition}
\label{PDinde}
Let $\mathscr{N}$ be the following CRN for $k$-site distributive phosphorylation/ dephosphorylation:
\[\begin{array}{c}
{S_0} + K \mathbin{\lower.3ex\hbox{$\buildrel\textstyle\rightarrow\over
{\smash{\leftarrow}\vphantom{_{\vbox to.5ex{\vss}}}}$}} {S_0}K{\mkern 1mu}  \to {S_1} + K{\mkern 1mu}  \mathbin{\lower.3ex\hbox{$\buildrel\textstyle\rightarrow\over
{\smash{\leftarrow}\vphantom{_{\vbox to.5ex{\vss}}}}$}} {S_1}K \to {S_2} + K{\mkern 1mu}  \mathbin{\lower.3ex\hbox{$\buildrel\textstyle\rightarrow\over
{\smash{\leftarrow}\vphantom{_{\vbox to.5ex{\vss}}}}$}} ... \to {S_k} + K\\
{S_k} + F \mathbin{\lower.3ex\hbox{$\buildrel\textstyle\rightarrow\over
{\smash{\leftarrow}\vphantom{_{\vbox to.5ex{\vss}}}}$}} ...{\mkern 1mu} {\mkern 1mu}  \to {\mkern 1mu} {S_2} + F \mathbin{\lower.3ex\hbox{$\buildrel\textstyle\rightarrow\over
{\smash{\leftarrow}\vphantom{_{\vbox to.5ex{\vss}}}}$}} {S_2}F \to {S_1} + F \mathbin{\lower.3ex\hbox{$\buildrel\textstyle\rightarrow\over
{\smash{\leftarrow}\vphantom{_{\vbox to.5ex{\vss}}}}$}} {S_1}F \to {S_0} + F
\end{array}\]
then 
\begin{itemize}
\item[i.] the fundamental classes generating the $\mathscr{F}$-decomposition are of the form \[\begin{array}{*{20}{c}}
{{S_i} + K \mathbin{\lower.3ex\hbox{$\buildrel\textstyle\rightarrow\over
{\smash{\leftarrow}\vphantom{_{\vbox to.5ex{\vss}}}}$}} {S_i}K \to {S_{i + 1}} + K}\\
{{S_{i + 1}} + F \mathbin{\lower.3ex\hbox{$\buildrel\textstyle\rightarrow\over
{\smash{\leftarrow}\vphantom{_{\vbox to.5ex{\vss}}}}$}} {S_{i + 1}}F \to {S_i} + F}
\end{array}{\text{ \ \ for }}i = 0,1,...,k - 1, {\text{ and }}\]
\item[ii.] the $\mathscr{F}$-decomposition is independent.
\end{itemize}
\end{proposition}
\begin{proof}
We choose the orientation consisting of the forward reactions and obtain:
\[\begin{array}{*{20}{c}}
{{\alpha _0}\left( {{S_0}K - \left( {{S_0} + K} \right)} \right) + {\beta _0}\left( {{S_1} + K - {S_0}K} \right) + {\alpha _1}\left( {{S_1}K - \left( {{S_1} + K} \right)} \right) + {\beta _1}\left( {{S_2} + K - {S_1}K} \right)}\\
{ + ... + {\alpha _{k - 1}}\left( {{S_{k - 1}}K - \left( {{S_{k - 1}} + K} \right)} \right) + {\beta _{k - 1}}\left( {{S_k} + K - {S_{k - 1}}K} \right) + }\\
{{\lambda _{k - 1}}\left( {{S_k}F - \left( {{S_k} + F} \right)} \right) + {\gamma _{k - 1}}\left( {{S_{k - 1}} + F - {S_k}F} \right) + ... + }\\
{{\lambda _1}\left( {{S_2}F - \left( {{S_2} + F} \right)} \right) + {\gamma _1}\left( {{S_1} + F - {S_2}F} \right) + {\lambda _0}\left( {{S_1}F - \left( {{S_1} + F} \right)} \right) + {\gamma _0}\left( {{S_0} + F - {S_1}F} \right) = 0.}
\end{array}\]
Hence, we have
\[\begin{array}{c}
{S_0}K\left( {{\alpha _0} - {\beta _0}} \right) + {S_0}\left( { - {\alpha _0} + {\gamma _0}} \right) + {S_1}K\left( {{\alpha _1} - {\beta _1}} \right) + ... + {S_{k - 1}}K\left( {{\alpha _{k - 1}} - {\beta _{k - 1}}} \right) + \\
{S_k}F\left( {{\lambda _{k - 1}} - {\gamma _{k - 1}}} \right) + ... + {S_2}F\left( {{\lambda _1} - {\gamma _1}} \right) + {S_1}F\left( {{\lambda _0} - {\gamma _0}} \right) + \\
K\left( { - {\alpha _0} + {\beta _0} - {\alpha _1} + {\beta _1} - ... - {\alpha _{k - 1}} + {\beta _{k - 1}}} \right) + F\left( { - {\lambda _{k - 1}} + {\gamma _{k - 1}} - ... - {\lambda _1} + {\gamma _1} - {\lambda _0} + {\gamma _0}} \right)\\
 + {S_1}\left( {{\beta _0} - {\alpha _1} + {\gamma _1} - {\lambda _0}} \right) + {S_2}\left( {{\beta _1} - {\alpha _2} + {\gamma _2} - {\lambda _1}} \right) + ... + {S_k}\left( {{\beta _{k - 1}} - {\lambda _{k - 1}}} \right) = 0.
\end{array}\]
Further manipulation gives ${\alpha _i} = {\beta _i} = {\lambda _i} = {\gamma _i}$ for each $i=0,...,k-1$, which yields (i). In addition, each of the stoichiometric subspaces of the $k$ subnetworks has dimension 3. Since the dimension of the stoichiometric subspace of the whole network is $3k$, (ii) holds.
\end{proof}

\begin{remark}
It is interesting to note that the linkage classes of the distributive CRN are not independent, in contrast to the $\mathscr{F}$-decomposition. The subnetworks of the latter are potentially useful for determining positive equilibria for any kinetics.
The review \cite{CODH2018} also presents models for dual-site PD with ERK mechanism and mixed processive-distributive mechanism. To complete the picture, we also determine their $\mathscr{F}$-decompositions in the following examples.
\end{remark}

\begin{example}
The CRN of dual-site PD with the ERK mechanism is given by:
\[\begin{array}{c}
{S_{00}} + K \mathbin{\lower.3ex\hbox{$\buildrel\textstyle\rightarrow\over
{\smash{\leftarrow}\vphantom{_{\vbox to.5ex{\vss}}}}$}} {S_{00}}K \to {S_{01}}K \to {S_{11}} + K{\rm{ \ \ \ \ \   }}{S_{11}} + F \mathbin{\lower.3ex\hbox{$\buildrel\textstyle\rightarrow\over
{\smash{\leftarrow}\vphantom{_{\vbox to.5ex{\vss}}}}$}} {S_{11}}F \to {S_{10}}F \to {S_{00}} + F\\
{S_{01}}K \mathbin{\lower.3ex\hbox{$\buildrel\textstyle\rightarrow\over
{\smash{\leftarrow}\vphantom{_{\vbox to.5ex{\vss}}}}$}} {S_{01}} + K{\rm{ \ \ \ \ \  }}{S_{10}}F \mathbin{\lower.3ex\hbox{$\buildrel\textstyle\rightarrow\over
{\smash{\leftarrow}\vphantom{_{\vbox to.5ex{\vss}}}}$}} {S_{10}} + F\\
{S_{10}} + K \mathbin{\lower.3ex\hbox{$\buildrel\textstyle\rightarrow\over
{\smash{\leftarrow}\vphantom{_{\vbox to.5ex{\vss}}}}$}} {S_{10}}K \to {S_{11}} + K{\rm{  \ \ \ \ \  }}{S_{01}} + F \mathbin{\lower.3ex\hbox{$\buildrel\textstyle\rightarrow\over
{\smash{\leftarrow}\vphantom{_{\vbox to.5ex{\vss}}}}$}} {S_{01}}F \to {S_{00}} + F
\end{array}\]
in which the $\mathscr{F}$-decomposition consists only of the whole network and hence independent.
\end{example}

\begin{example}
The CRN of dual-site PD with the mixed-mode mechanism is given by:
\[\begin{array}{c}
{S_0} + K \mathbin{\lower.3ex\hbox{$\buildrel\textstyle\rightarrow\over
{\smash{\leftarrow}\vphantom{_{\vbox to.5ex{\vss}}}}$}} {S_0}K \to {S_1}K \to {S_2} + K\\
{S_2} + F \mathbin{\lower.3ex\hbox{$\buildrel\textstyle\rightarrow\over
{\smash{\leftarrow}\vphantom{_{\vbox to.5ex{\vss}}}}$}} {S_2}F \to {S_1} + F \mathbin{\lower.3ex\hbox{$\buildrel\textstyle\rightarrow\over
{\smash{\leftarrow}\vphantom{_{\vbox to.5ex{\vss}}}}$}} {S_1}F \to {S_0} + F
\end{array}\]
in which the $\mathscr{F}$-decomposition consists only of the whole network and hence independent.
\end{example}

\subsection{Incidence-independent $\mathscr{F}$-decompositions}
This section begins with a necessary condition for an incidence-independent $\mathscr{F}$-decomposition analogous to the proposition in Section \ref{necessarycondition:independent} for an independent $\mathscr{F}$-decomposition. We then present a sufficient condition for incidence-independence, namely when the $\mathscr{F}$-decomposition is a $\mathscr{C}$-decomposition, and show that various subsets of S-system CRNs fulfill the condition. Finally, in Section \ref{sect:PDinciinde}, we show that the CRNs of PD processes are also incidence-independent (and hence bi-independent).

\subsubsection{A necessary condition for incidence-independent $\mathscr{F}$-decompositions}
The following proposition is the analogue of Proposition \ref{boundforw} in Section \ref{necessarycondition:independent}:
\begin{proposition}
Let $\mathscr{N} = \mathscr{N}_1 \cup \mathscr{N}_2 \cup ... \cup \mathscr{N}_k$ be a CRN decomposition. If the decomposition is incidence-independent, then $k \le n-l$. If $\mathscr{N}$ has zero deficiency, then $k \le s$.
\end{proposition}
\begin{proof}
Since $n_i - l_i \ge 1$, we have $\sum \left(n_i - l_i\right) \ge k$. If $k > n - l$, then  $\sum \left(n_i - l_i \right)  > n - l$, i.e., the decomposition is incidence-dependent. If the deficiency is zero, $n - l = s$.
\end{proof}

\begin{corollary}
For an incidence-independent $\mathscr{F}$-decomposition, $w \le n - l$. If $\mathscr{N}$ has zero deficiency, then $w\le s$.
\end{corollary}

\begin{proof}
For the $\mathscr{F}$-decomposition, $w = k - 1$ (if $P_0$ is non-empty) or $w = k$ (otherwise), and the claims follow.
\end{proof}

Unfortunately, the condition does not seem to be as useful as its analogue as the following example shows.
\begin{example}
The CRN of a popular Generalized Mass Action (GMA) model of anaerobic yeast fermentation (denoted by ERM0-G) was analyzed in \cite{hernandez} for its capacity for multistationarity. The computation of its $\mathscr{F}$-decomposition showed that $w = 11 < 12 = n - l$, so that the stated condition is fulfilled. However, the same computation showed that the sum of $n_i - l_i$ (over the 11 subnetworks) = 13, so that the $\mathscr{F}$-decomposition is not incidence-independent.
\end{example}

\subsubsection{A sufficient condition: when the $\mathscr{F}$-decomposition is a $\mathscr{C}$-decomposition}
\label{sufficientconditionfcdec}

The $\mathscr{F}$-decompositions of S-system CRNs, though always independent, are in general not incidence-independent. To show this, we recall an example from \cite{FML2018}.
\begin{example}
The (embedded) S-system CRN of a model of the gene regulatory system of Mycobacterium tuberculosis (Mtb) in the non-replicating phase (NRP) of its life cycle was shown to have $m = 40$ species, $n = 98$ complexes, $r =  80$ irreversible reactions and $l = 19$ linkage classes. Since for any (embedded) S-system CRN, the network rank $s = m$, the deficiency $\delta = 98 - 19 - 40 = 39 < 40$. Theorem 1 in \cite{FML2018} implies that the $\mathscr{F}$-decomposition, i.e., species decomposition, is incidence-independent.
\end{example}

However, various subsets of S-system CRNs always have incidence-independent $\mathscr{F}$-decompositions.  Before discussing these examples, we show that in many cases, the incidence-independence is due to the fact that the $\mathscr{F}$-decompositions are $\mathscr{C}$-decompositions. We recall the following proposition and proof from \cite{FML2018}:

\begin{proposition}
Any $\mathscr{C}$-decomposition is incidence-independent.
\end{proposition}

\begin{proof}
Theorem 5 in \cite{FML2018} characterizes $\mathscr{C}$-decompositions as follows: a decomposition is a $\mathscr{C}$-decomposition if and only if it is a coarsening of the linkage class decomposition. Basically, this means that any subnetwork of a $\mathscr{C}$-decomposition is the disjoint union of (some) linkage classes. Since it is also shown (Proposition 3 of \cite{FML2018}) that any coarsening of an incidence-independent decomposition is incidence-independent, it follows that any $\mathscr{C}$-decomposition is incidence-independent.
\end{proof}

\begin{corollary}
If the $\mathscr{F}$-decomposition of a CRN is a $\mathscr{C}$-decomposition, then $2w \le n$.
\end{corollary}

\begin{proof}
For any $\mathscr{C}$-decomposition, the number of subnetworks $k \le l$, hence $w \le l$.  Since a $\mathscr{C}$-decomposition is incidence-independent, we have $w + l \le n$. Combining the two inequalities shows the claim.
\end{proof}

\begin{example}
The $\mathscr{F}$-decomposition of any (embedded) S-system CRN with m irreversible species and distinct complexes consists of subnetworks with 2 linkage classes containing the inflow and outflow reaction for each species respectively, and hence is a $\mathscr{C}$-decomposition. Since each CRN has $4m$ complexes, we will denote this subset by Ssys$_{4m}$.
\end{example}

\begin{example}
For any (embedded) S-system CRN with m reversible species, the $\mathscr{F}$-decomposition coincides with the linkage class decomposition. Each such CRN has $2m$ complexes and the subset will be denoted by Ssys$_{2m}$.
\end{example}

A further class of CRNs whose $\mathscr{F}$-decompositions are $\mathscr{C}$-decompositions is provided by a special case of Theorem \ref{cycles:graph} in Section \ref{decomposition:types:network:prop}.

\begin{example}
Let the CRN $\mathscr{N} =\{\mathscr{N}_i | \mathscr{N}_i = (\mathscr{C}_i, \mathscr{R}_i)\}$ with sequence of long monomolecular directed cycles, i.e., of length 
$\ge$ 3, and $|\mathscr{C}_i \cap \mathscr{C}_j| = 0$ for distinct $i,j = 1,2,...,k$.
\end{example}

It is shown in Theorem \ref{cycles:graph}, that the $\mathscr{F}$-decomposition consists precisely of the $C_i$, which are simultaneously the linkage classes of the network.

We now discuss the CRN of the Heck et al. model of ``terrestrial carbon recovery'' from \cite{hernandez}.

\begin{example}
The CRN is given by:
$$ \begin{array}{lll}
R_1: A_1 +2A_2 \to 2A_1 + A_2 & \ \ \ & R_6: A_1 + 2A_4  \to 2A_1 + A_4\\
R_2: A_1 + A_2  \to 2A_2  & \ \ \ & R_7: A_1 + A_4  \to 2A_4 \\
R_3: A_2  \to A_3 & \ \ \  & R_8: A_4  \to A_3\\
R_4: A_3  \to A_2  & \ \ \  & R_9: A_3 \to A_4  \\
R_5: A_4 + A_5  \to 2A_4 & \ \ \  & R_{10}: A_1 + A_2 + A_4 \to A_5 + A_2+A_4\\
\end{array}$$
\end{example}

The $\mathscr{F}$-decomposition is a curiosity: it is almost a $\mathscr{C}$-decomposition, with 4 of 6 fundamental classes coinciding with the linkage class reaction sets $\mathscr{C}_1 =  \{R_1\}$, $\mathscr{C}_2 = \{R_2\}$, $\mathscr{C}_3 = \{R_3, R_4, R_8, R_9\}$, and $\mathscr{C}_4 = \{ R_6\}$. For the remaining two, $\mathscr{C}_5 = \{R_5, R_{10}\} \ne \{R_5, R_7\}$ and $\mathscr{C}_6 = \{R_7\} \ne \{R_{10}\}$, but fortuitously,  the images of their incidence maps are respectively isomorphic. Hence the $\mathscr{F}$-decomposition is also incidence-independent.

As the final example in this section, we present a subset of S-system CRNs whose $\mathscr{F}$-decompositions are always incidence-independent but which are not $\mathscr{C}$-decompositions.

\begin{example}
The set of (embedded) S-system CRNs with self-regulating species and non-regulated outflows have $\mathscr{F}$-subnetworks of the form $\left\{ {0 \leftarrow {x_i} \to 2{x_i}} \right\}$, i.e., there are $2m+1$ complexes and one linkage class, inferring $n - l = 2m$. Clearly, it is not a $\mathscr{C}$-decomposition. Since each $\mathscr{F}$-subnetwork has 3 complexes and a linkage class, the sum of $n_i - l_i = 2m$, too, showing incidence-independence. We denote this subset as Ssys$_{2m+1}$.
\end{example}

\subsubsection{CRNs of phosphorylation/dephosphorylation systems have incidence-independent $\mathscr{F}$-decompositions}
\label{sect:PDinciinde}
In this section, we use the computations for $\mathscr{F}$-decompositions of the CRNs of PD systems in Section \ref{PD:independent} to show that they are also incidence-independent (hence bi-independent). For multisite processive PD systems, this is trivial since the $\mathscr{F}$-decomposition has only one subnetwork.  The cases of dual-site ERK mechanism and mixed-mechanism are discussed in the following example.

\begin{example}
The CRN of dual-site PD with the mixed-mode mechanism is given by:
\[\begin{array}{c}
{S_0} + K \mathbin{\lower.3ex\hbox{$\buildrel\textstyle\rightarrow\over
{\smash{\leftarrow}\vphantom{_{\vbox to.5ex{\vss}}}}$}} {S_0}K \to {S_1}K \to {S_2} + K\\
{S_2} + F \mathbin{\lower.3ex\hbox{$\buildrel\textstyle\rightarrow\over
{\smash{\leftarrow}\vphantom{_{\vbox to.5ex{\vss}}}}$}} {S_2}F \to {S_1} + F \mathbin{\lower.3ex\hbox{$\buildrel\textstyle\rightarrow\over
{\smash{\leftarrow}\vphantom{_{\vbox to.5ex{\vss}}}}$}} {S_1}F \to {S_0} + F
\end{array}\]
in which the $\mathscr{F}$-decomposition consists only of the whole network and hence incidence-independent.
\end{example}

The following proposition completes the picture by providing the proof in the multisite distributive PD CRN to be incidence-independent.

\begin{proposition}
\label{PDinciinde}
The $k$-site distributive phosphorylation/dephosphorylation:
\[\begin{array}{c}
{S_0} + K \mathbin{\lower.3ex\hbox{$\buildrel\textstyle\rightarrow\over
{\smash{\leftarrow}\vphantom{_{\vbox to.5ex{\vss}}}}$}} {S_0}K{\mkern 1mu}  \to {S_1} + K{\mkern 1mu}  \mathbin{\lower.3ex\hbox{$\buildrel\textstyle\rightarrow\over
{\smash{\leftarrow}\vphantom{_{\vbox to.5ex{\vss}}}}$}} {S_1}K \to {S_2} + K{\mkern 1mu}  \mathbin{\lower.3ex\hbox{$\buildrel\textstyle\rightarrow\over
{\smash{\leftarrow}\vphantom{_{\vbox to.5ex{\vss}}}}$}} ... \to {S_k} + K\\
{S_k} + F \mathbin{\lower.3ex\hbox{$\buildrel\textstyle\rightarrow\over
{\smash{\leftarrow}\vphantom{_{\vbox to.5ex{\vss}}}}$}} ...{\mkern 1mu} {\mkern 1mu}  \to {\mkern 1mu} {S_2} + F \mathbin{\lower.3ex\hbox{$\buildrel\textstyle\rightarrow\over
{\smash{\leftarrow}\vphantom{_{\vbox to.5ex{\vss}}}}$}} {S_2}F \to {S_1} + F \mathbin{\lower.3ex\hbox{$\buildrel\textstyle\rightarrow\over
{\smash{\leftarrow}\vphantom{_{\vbox to.5ex{\vss}}}}$}} {S_1}F \to {S_0} + F
\end{array}\]
has incidence-independent $\mathscr{F}$-decomposition.
\end{proposition}
\begin{proof}
In Proposition \ref{PDinde}, it was shown that the fundamental classes generating the $\mathscr{F}$-decomposition are of the form \[\begin{array}{*{20}{c}}
{{S_i} + K \mathbin{\lower.3ex\hbox{$\buildrel\textstyle\rightarrow\over
{\smash{\leftarrow}\vphantom{_{\vbox to.5ex{\vss}}}}$}} {S_i}K \to {S_{i + 1}} + K}\\
{{S_{i + 1}} + F \mathbin{\lower.3ex\hbox{$\buildrel\textstyle\rightarrow\over
{\smash{\leftarrow}\vphantom{_{\vbox to.5ex{\vss}}}}$}} {S_{i + 1}}F \to {S_i} + F}
\end{array}{\text{ for }}i = 0,1,...,k - 1.\]
Since the network has $2(2k)=4k$ reactions, and the reactants and products are distinct, it has $4k+2$ complexes. Now, \[\sum {\left( {{n_i} - {l_i}} \right)}  = \sum {\left( {6 - 2} \right)}  = 4k = \left( {4k + 2} \right) - 2 = n - l.\]
\end{proof}

Clearly, the image of $I_{a,\mathscr{O}}$, the restriction of the incidence map to $R^\mathscr{O}$, is equal to the image of $I_a$, since the image of a reverse reaction is just the negative of the forward reaction and the linkage classes are not changed. Hence,  $r_{\text{irr}} + r_{\text{rev}} - (n - l) \ge 0$, being the $\text{dim } \text{Ker} I_{a,\mathscr{O}}$, or $n - l \le r_{\text{irr}} + r_{\text{rev}}$. As shown in propositions above, for independent and incidence-independent $\mathscr{F}$-decompositions, $w \le n - l$. To date, we have not yet found a CRN with a dependent and incidence-dependent $\mathscr{F}$-decomposition with $w > n - l$, so one can ask the question whether $n - l$ is an upper bound for $w$ in general. 

\section{Types of $\mathscr{F}$-decomposition and Network Properties}
\label{decomposition:types:network:prop}
Ji classified the subnetworks occurring in a $\mathscr{P}$-decomposition into 3 types and summarized their properties as follows (Proposition 2.5.4 in \cite{ji}).

\begin{lemma} {\bf \cite{ji}}
Let $\mathscr{N}=\left(\mathscr{S},\mathscr{C},\mathscr{R}\right)$ be a CRN and $\mathscr{O}$ be an orientation. Let $\mathscr{{N}}_{\mathscr{O},i}$ for $i=0,1,2,...,w$ be defined as the subnetwork generated by all reactions in $P_i$. Then one of the following holds:
\begin{itemize}
\item [i.] The reaction vectors for $\mathscr{{N}}_{\mathscr{O},i}$ are linearly independent, and the subnetwork $\mathscr{{N}}_{\mathscr{O},i}$ based on $P_i$ forms a forest (i.e., a graph with no cycle) with deficiency 0.
\item [ii.] The reaction vectors are minimally dependent, and the subnetwork $\mathscr{{N}}_{\mathscr{O},i}$ based on $P_i$ forms a forest with deficiency 1.
\item [iii.] The reaction vectors are minimally dependent, and the subnetwork $\mathscr{{N}}_{\mathscr{O},i}$ based on $P_i$ forms a big cycle (with at least three vertices) with deficiency 0.
\end{itemize}
\label{ji254}
\end{lemma}

We will denote the subnetwork classes in i, ii, and iii of Lemma \ref{ji254} as Type I, Type II and Type III subnetworks respectively. Since our focus is on the $\mathscr{F}$-decomposition, we extend this classification as follows: an $\mathscr{F}$-subnetwork is of type I, II or III if it contains a $\mathscr{P}$-subnetwork of type I, II or III, respectively. Note that while the  characterization of Type I and II $\mathscr{P}$-subnetworks as forests is lost, that Type III subnetwork as a big cycle is retained in the Type III $\mathscr{F}$-subnetwork. More importantly, the deficiency of each subnetwork type remains the same.  We assign the numbers of fundamental classes for Types I, II and III with the symbols $w_I$, $w_{II}$ and $w_{III}$, respectively. 

\begin{definition}
An $\mathscr{F}$-decomposition is said to be
\begin{itemize}
\item[i.] {\bf Type I} if it contains Type I subnetwork only.
\item[ii.] {\bf Type II} if it contains Type II subnetwork only.
\item[iii.] {\bf Type III} if it contains Type III subnetwork only.
\end{itemize}
\end{definition}

The first important consequence of the above classification is the following:

\begin{proposition}
If a CRN has an independent $\mathscr{F}$-decomposition, then its deficiency $\delta \le w_{II}$.
\end{proposition}

\begin{proof}
Note that the $\mathscr{F}$-decomposition is independent, so $\delta \le \delta _1 + \delta _2 ... + \delta _w$. Since any Type I or Type III subnetworks have zero deficiency, we obtain the claim.
\end{proof}

\begin{corollary}
A CRN whose $\mathscr{F}$-decomposition has no Type II subnetworks and is independent has zero deficiency. In particular, this is true for CRNs with Type I and Type III independent $\mathscr{F}$-decompositions.
\label{cor:zero:def}
\end{corollary}

The special case of the corollary above is significant because the properties of the positive equilibria sets of deficiency zero networks are well-known. In particular:

\begin{itemize}
\item[i.] any positive equilibrium is complex-balanced \cite{feinberg_DOA},
\item[ii.] no positive equilibrium exists if the network is not weakly reversible (Deficiency Zero Theorem), and
\item[iii.] for some subsets of the power-law kinetics, existence and parametrization results on equilibria are proven, i.e., Deficiency Zero Theorem for MAK (Feinberg-Horn-Jackson), PL-RDK with zero kinetic deficiency (M${\rm\ddot{u}}$ller-Regensburger \cite{MURE2014}), PL-TIK (Talabis et al. \cite{TAM2018}), and PL-NDK with special independent decompositions (Fortun et al. \cite{fortun2}).
\end{itemize}

\subsection{Independent Type I $\mathscr{F}$-decompositions}
\begin{proposition}
Let $s$ be the rank of a network $\mathscr{N}$. If for an orientation $\mathscr{O}$, $s=|\mathscr{O}|$, then $\delta \left( \mathscr{N} \right)=0$.
\end{proposition}

\begin{proof}
Let $s$ be the rank of a network $\mathscr{N}$ with an orientation $\mathscr{O}$. Note that $s=\dim S$. Each reaction in $\mathscr{O}$ has a corresponding reaction vector. With the assumption that $s=|\mathscr{O}|$, the reaction vectors are linearly independent. From Lemma \ref{ji254}, there are Type I subnetworks only in the $\mathscr{P}$-decomposition. By definition, the $\mathscr{P}$-decomposition is of Type I. From Corollary \ref{cor:zero:def}, where the number of Type II subnetworks in the $\mathscr{P}$-decomposition is zero, it follows that $\delta \left( \mathscr{N} \right)=0$.
\end{proof}

\begin{example}
If the network has only irreversible reactions, then the only orientation is the whole set of reactions. In this case, a network with independent Type I $\mathscr{P}$-decomposition is a trivial nullspace network.
\end{example}

\begin{example}
The CRNs of the set Ssys$_{2m}$ introduced in Section \ref{sufficientconditionfcdec} all have Type I $\mathscr{F}$-decompositions.
\end{example}

\begin{proposition}
Let $\left(\mathscr{S},\mathscr{C},\mathscr{R},K\right)$ be a PL-RDK system such that there is at least one irreversible reaction in $\mathscr{R}$. Let $\mathscr{{N}}_{\mathscr{O},i}$ for $i=0,1,2,...,w$ be defined as the subnetwork generated by all reactions in $P_i$. If the $\mathscr{P}$-decomposition is independent, and the reaction vectors in $\mathscr{{N}}_{\mathscr{O},i}$ are linearly independent for each $i$, then the system does not have the capacity to admit multiple equilibria.
\label{independent_multiple}
\end{proposition}

\begin{proof}
Suppose the $\mathscr{P}$-decomposition is independent and the reaction vectors in $\mathscr{{N}}_{\mathscr{O},i}$ are linearly independent for each $i$. Thus, $Ker{L_\mathscr{O}}$ is trivial (containing the zero vector only). So every reaction must be placed in the zeroth equivalence class $P_0$. But there is one irreversible reaction contradicting the rule in the higher deficiency algorithm \cite{hernandez,ji} that each reaction in $P_0$ must be reversible (with respect to $\mathscr{R}$). Therefore, the system does not have the capacity to admit multiple equilibria.
\end{proof}

\subsection{Independent Type II $\mathscr{F}$-decompositions}

The following proposition expresses an apparently rare relationship between deficiency and rank of a CRN.

\begin{proposition}
For a CRN with an independent Type II $\mathscr{F}$-decomposition, $\delta \le s$ or equivalently, $s \le n - l \le 2s$.
\label{inde:typeII}
\end{proposition}

\begin{proof}
We have for such a CRN, $\delta \le w_{II} = w \le s$.
\end{proof}

\begin{example}
The embedded CRN of an S-system with only irreversible species is an example of an independent Type II $\mathscr{F}$-decomposition network. This shows that the formula in \cite{FML2018} can be seen in this case as an instance of Proposition \ref{inde:typeII}. If the $\mathscr{F}$-decomposition is also incidence-independent, then the network deficiency is the number of species $m$. The sets Ssys$_{4m}$ and Ssys$_{2m+1}$ discussed in Section \ref{sufficientconditionfcdec} are subsets of this set of S-system CRNs.
\end{example}

\begin{example}
The following CRN is the well-known model of the EnvZ-OmpR system of E. coli studied by Shinar and Feinberg in \cite{SHFE2010}:
\[\begin{array}{c}
X \mathbin{\lower.3ex\hbox{$\buildrel\textstyle\rightarrow\over
{\smash{\leftarrow}\vphantom{_{\vbox to.5ex{\vss}}}}$}} XT \to {X_p}\\
{X_p} + Y \mathbin{\lower.3ex\hbox{$\buildrel\textstyle\rightarrow\over
{\smash{\leftarrow}\vphantom{_{\vbox to.5ex{\vss}}}}$}} {X_p}Y \to X + {Y_p}\\
XT + {Y_p} \mathbin{\lower.3ex\hbox{$\buildrel\textstyle\rightarrow\over
{\smash{\leftarrow}\vphantom{_{\vbox to.5ex{\vss}}}}$}} XT{Y_p} \to XT + Y
\end{array}\]
The $\mathscr{F}$-decomposition, like that of the multisite processive PD model, has only one subnetwork (i.e., the whole CRN) and is of Type II (its $\mathscr{F}$-subnetwork is clearly a forest of deficiency 1).
\end{example}

\begin{example}
In \cite{TMMNJ19} and \cite{VELO2014}, it is shown that evolutionary games with replicator dynamics can be represented as chemical kinetic systems. The following example is a symmetric population game with 2 pure strategies and non-linear, continuous payoff functions, also called ``playing the field'' games  \cite{CRTA2015}. 
Let $x = (x_1,x_2)$ be the vector of pure strategies and $F$ a $2 \times 2$ matrix of nonnegative real numbers with $F_i$ as its $i$-th row. The payoff function $f_i$ is defined as $f_i(x) = x^{F_i}$. A PLK representation is then given by the CRN:
\[ \begin{array}{l}
R_i: x_i \to 2x_i\\
R_{-i}: x_i \to 0\\
R_{-i}': 2x_i \to x_i
\end{array}\]
The rate functions  for the reactions are given by $K_i(x) = x_if_i(x)$, $K_{-i}'(x) = x_ix_1f_1(x)$ and $K_{-i}(x) = x_ix_2f_2(x)$. The rate constants are set to 1 to ensure dynamic equivalence with the replicator equation.

To determine the $\mathscr{F}$-decomposition, we consider the orientation given by $\{ R_i, R_{-i}\}$. The subnetwork generated by this orientation is an S-system in 2 irreversible variables, and its $\mathscr{F}$-decomposition is independent. Since a $\mathscr{P}$-decomposition of the game's CRN is independent, then its $\mathscr{F}$-decomposition is independent too.
\end{example}

\subsection{Independent Type III $\mathscr{F}$-decompositions}
\begin{runningexample}
The subnetwork of the Schmitz's  carbon cycle model is clearly an instance of an independent Type III $\mathscr{F}$-decomposition. We formulate a generalization in the following theorem:
\end{runningexample}

\begin{theorem}
\label{cycles:graph}
The following family of CRNs has bi-independent Type III $\mathscr{F}$-decomposition such that the $\mathscr{N}_i$'s are precisely the fundamental classes under the decomposition:
$\mathscr{N} =\{\mathscr{N}_i | \mathscr{N}_i = (\mathscr{C}_i, \mathscr{R}_i)\}$ with a (possibly broken) chain of long monomolecular directed cycles, i.e. of length 
$\ge$ 3, and $|\mathscr{C}_i \cap \mathscr{C}_j| \le 1$ if $j = i + 1$ for $i = 0,1,...,k-1$.
\end{theorem}

\begin{figure*}
\begin{center}
\includegraphics[width=14cm,height=28cm,keepaspectratio]{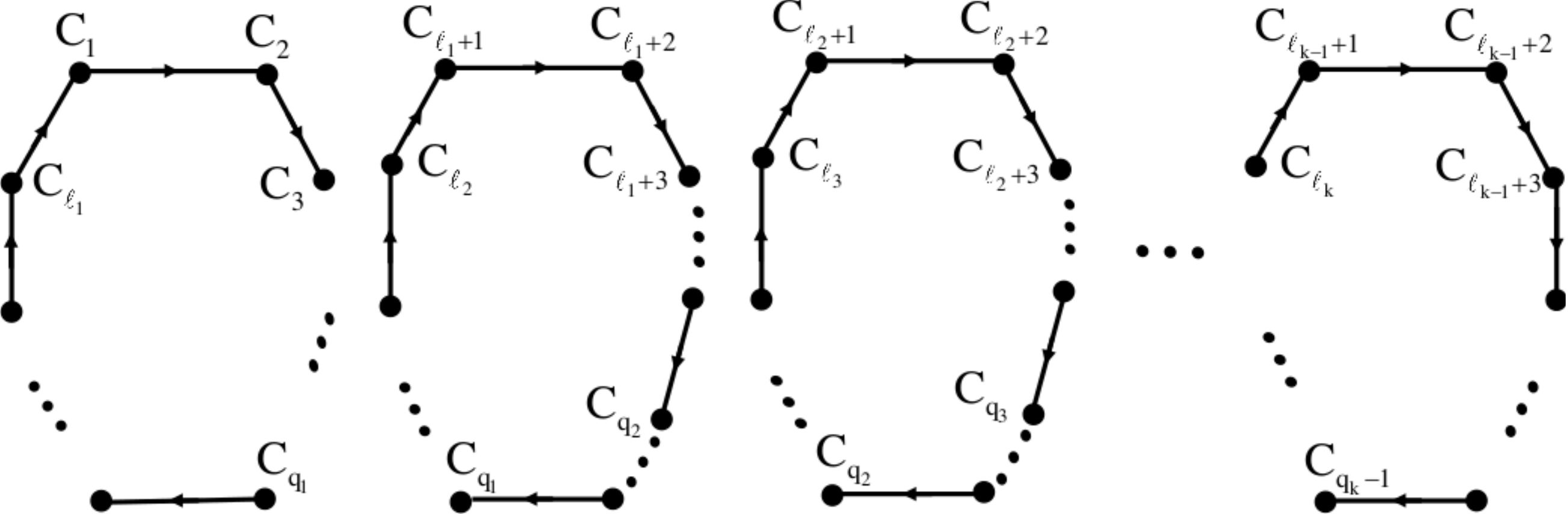}
\caption{An illustration of the graph with no break in Theorem \ref{cycles:graph}.}
\label{cyclesgraph1}
\end{center}
\end{figure*}

\begin{proof}
To simplify the proof, we will only show the case with no break, since the one with break in the graph is rather obvious. Without loss of generality, we assume the orientation given by the graph in Figure \ref{cyclesgraph1}. Note that there are exactly $\ell=\ell_1+\ell_2+...+\ell_k-(k-1)$ complexes in the network. In solving a basis for ${\rm Ker \ } L_\mathscr{O}$, we have the following equation:
\[{\alpha _1}\left( {{C_2} - {C_1}} \right) + {\alpha _2}\left( {{C_3} - {C_2}} \right) + ... + {\alpha _{{\ell _1}}}\left( {{C_1} - {C_{{\ell _1}}}} \right) + \]
\[{\alpha _{{\ell _1} + 1}}\left( {{C_{{\ell _1} + 2}} - {C_{{\ell _1} + 1}}} \right) + {\alpha _{{\ell _1} + 2}}\left( {{C_{{\ell _1} + 3}} - {C_{{\ell _1} + 2}}} \right) + ... + {\alpha _{{\ell _2}}}\left( {{C_{{\ell _1} + 1}} - {C_{{\ell _2}}}} \right) + ... + \]
\[{\alpha _{{\ell _{k - 1}} + 1}}\left( {{C_{{\ell _{k - 1}} + 2}} - {C_1}} \right) + {\alpha _{{\ell _{k - 1}} + 2}}\left( {{C_{{\ell _{k - 1}} + 3}} - {C_{{\ell _{_{k - 1}}} + 2}}} \right) + ... + {\alpha _{{\ell _k}}}\left( {{C_{{\ell _{_{k - 1}}} + 1}} - {C_{{\ell _k}}}} \right) = 0.\]
Hence, we have:
\[{C_1}\left( {{\alpha _{{\ell _1}}} - {\alpha _1}} \right) + {C_2}\left( {{\alpha _1} - {\alpha _2}} \right) + {C_3}\left( {{\alpha _2} - {\alpha _3}} \right) + ... + {C_{{\ell _1}}}\left( {{\alpha _{{\ell _1} - 1}} - {\alpha _{{\ell _1}}}} \right) + \]
\[{C_{{\ell _1} + 1}}\left( {{\alpha _{{\ell _2}}} - {\alpha _{{\ell _1} + 1}}} \right) + {C_{{\ell _1} + 2}}\left( {{\alpha _{{\ell _1} + 1}} - {\alpha _{{\ell _1} + 2}}} \right) + ... + {C_{{\ell _2}}}\left( {{\alpha _{{\ell _2} - 1}} - {\alpha _{{\ell _2}}}} \right) + ... + \]
\[{C_{{\ell _{_{k - 1}}} + 1}}\left( {{\alpha _{{\ell _k}}} - {\alpha _{{\ell _{k - 1}} + 1}}} \right) + {C_{{\ell _{k - 1}} + 2}}\left( {{\alpha _{{\ell _{k - 1}} + 1}} - {\alpha _{{\ell _{k - 1}} + 2}}} \right) + ... + {C_{{\ell _k}}}\left( {{\alpha _{{\ell _k} - 1}} - {\alpha _{{\ell _k}}}} \right) = 0.\]
For the first subnetwork, ${\alpha _{{\ell _1}}} = {\alpha _1} = {\alpha _2} = ... = {\alpha _{{q_1} - 1}}$ and 
${\alpha _{{q_1}}} = {\alpha _{{q_1} + 1}} = ... = {\alpha _{{\ell _1}}}$, which gives ${\alpha _1} = {\alpha _2} = ... = {\alpha _{{\ell _1}}}$.
A similar proof can be provided for the last subnetwork.
For the remaining subnetworks, with $i=2,3,...,k-1$, we obtain the summand:
\[.. .+ \left( {{\alpha _{{q_i} - 1}}} \right)'\left( {{C_{{q_i}}} - \left( {{C_{{q_i} - 1}}} \right)'} \right) + {\alpha _{{q_i}}}\left( {\left( {{C_{{q_i} + 1}}} \right)' - {C_{{q_i}}}} \right)  + \left( {{\alpha _{{q_i} + 1}}} \right)'\left( {\left( {{C_{{q_i} + 2}}} \right)' - \left( {{C_{{q_i} + 1}}} \right)'} \right) + ...\]
which yields
\[... + \left( {{C_{{q_i} - 1}}} \right)'\left( {\left( {{\alpha _{{q_i} - 2}}} \right)' - \left( {{\alpha _{{q_i} - 1}}} \right)'} \right) + ... + \left( {{C_{{q_i} + 1}}} \right)'\left( {\left( {{\alpha _{{q_i}}}} \right)' - \left( {{\alpha _{{q_i} + 1}}} \right)'} \right) + ....\]
Note that the ``apostrophe'' symbol is used to differentiate the positions of the complexes from two consecutive subnetworks. Now, the term $\left( {{C_{{q_i} - 1}}} \right)'\left( {\left( {{\alpha _{{q_i} - 2}}} \right)' - \left( {{\alpha _{{q_i} - 1}}} \right)'} \right)$ yields $\left( {{\alpha _{{q_i} - 2}}} \right)' = \left( {{\alpha _{{q_i} - 1}}} \right)'$ which implies the equality of the $\alpha_r$'s with position $r<q_i$. Similarly, the term $\left( {{C_{{q_i} + 1}}} \right)'\left( {\left( {{\alpha _{{q_i}}}} \right)' - \left( {{\alpha _{{q_i} + 1}}} \right)'} \right)$ yields $\left( {{\alpha _{{q_i}}}} \right)' - \left( {{\alpha _{{q_i} + 1}}} \right)'$ which implies the equality of $\alpha_r$'s with position $r>q_i$. 
We also obtain \[{C_{{q_i}}}\left( {\left( {{\alpha _{{q_i} - 1}}} \right)' - \left( {{\alpha _{{q_i}}}} \right)'} \right) + \left( {{C_{{q_i} + 1}}} \right)'\left( {\left( {{\alpha _{{q_i}}}} \right)' - \left( {{\alpha _{{q_i} + 1}}} \right)'} \right).\]
Since the complex $C_{q_i}$ is also present in the $i+1$-st subnetwork, we get
\[{C_{{q_i}}}\left( {\left( {{\alpha _{{q_i} - 1}}} \right)'' - \left( {{\alpha _{{q_i}}}} \right)''} \right) + \left( {{C_{{q_i} + 1}}} \right)''\left( {\left( {{\alpha _{{q_i}}}} \right)'' - \left( {{\alpha _{{q_i} + 1}}} \right)''} \right).\]
It follows that $\left( {{\alpha _{{q_i} + 1}}} \right)' = {\alpha _{{q_i}}}$.
But ${\alpha _{{\ell _i} + 1}} = {\alpha _{{\ell _{i + 1}}}}$, which proves the equality of the $\alpha$'s in a subnetwork. Thus, the subnetworks are precisely the fundamental classes which are independent. Indeed, the following is a basis for ${\rm Ker \ } L_\mathscr{O}$:
\[
 \begin{pmatrix} 
    F_{1} & &  & &  \\
     & F_{2} & &\text{\huge0} & \\
     & & \ddots &  & \\
     & \text{\huge0}& & \ddots & \\
     &  &   &   & F_{k} 
    \end{pmatrix}
\]
where each $F_i$ is an $\ell_i \times 1$ matrix with entries all equal to 1.

Since $\mathscr{N}$ has zero deficiency, this also proves the incidence-independence of the $\mathscr{F}$-decomposition.
\end{proof}

\section{The CF-RI$_+$ Transformation Method}
\label{sect:cfri:transform:method}
In this section, we present a transformation method whose key property is that it maps an irreversible reaction (a reversible pair of reactions) of the original system to an irreversible reaction (a reversible pair of reactions) of the target system. In other words, it is reversibility and irreversibility (RI) preserving.
This method was based on the generic CF-RM method (transformation of complex factorizable kinetics by reactant multiples)
which converts a PL-NDK to a PL-RDK system.
We add in the notation CF-RI  a sub-index ``+'' for two reasons: to indicate the ``positive'' (or preserving) relation and to highlight its partial coincidence with the CF-RM$_+$ variant of CF-RM. However, in most cases, CF-RI$_+$ adds new reactants which are not reactant multiples, so it is not a CF-RM variant. 

\subsection{Review of the CF-RM$_+$ Method}
\indent We present the CF-RM transformation method in \cite{cfrm}. One can construct a PL-RDK system from a given PL-NDK system using this method. A CF-subset contains reactions having the same kinetic order vectors. At each reactant complex, the branching reactions are partitioned into CF-subsets.
An NF-reactant complex has more than one CF-subset which makes the system NDK. For each subset, a complex is added to both the reactant and the product complexes of a reaction leaving the reaction vectors unchanged. The kinetic order matrix does not change as well.

The CF-RM method is given by the following steps.
\begin{itemize}
\item [1.] Determine the set of reactant complexes $\rho \left(\mathscr{R} \right)$.
\item [2.] Leave each CF-reactant complex unchanged.
\item [3.] At an NF-reactant complex, select a CF-subset containing the highest number of reactions and leave this CF-subset unchanged. For each of the remaining $N_R(y)-1$ CF-subsets, choose successively a multiple of $y$ which is not among the
current set of reactants. Different procedures are possible for the selection of a new reactant as long as
it is different from those in the current reactant set. After each choice, the current set is updated.
\end{itemize}

CF-RM$_+$ is a variant of CF-RM. All the steps are identical with the generic CF-RM method except that it uses additional criteria in the selection of the new reactant multiples. CF-RM$_+$ chooses the reactant multiple so that the new reactant differs from all existing complexes and all the new product complexes in the CF-subset also differ from all existing complexes \cite{cfrm}.

\subsection{Details of the CF-RI$_+$ Method}
Note that the CF-RM$_+$ method given in \cite{cfrm} updates the set of current complexes and complexes in the transform after each CF-subset of an NF-node is processed. The CF-RI$_+$ method proceeds as follows:
\begin{itemize}
\item[1.] Determine the reactant set $\rho({\mathscr{R})}$ and identify the subset $\rho({\mathscr{R})}_{CF}$ of CF-nodes.
\item[2.] If the reaction set ${\mathscr{R}_y:=\rho^{-1}(y)}$
of a CF-node $y$ has no reversible reaction with an NF-node, then it is left unchanged.  
\item[3.] At an NF-node without reversible reactions, carry out the steps of CF-RM$_+$.  
\item[4.] At an NF-node with a reversible reaction, among the CF-subsets without a reversible reaction (if there are any), select one with the highest number of reactions and leave this unchanged.
\item[5.] For the remaining CF-subsets without a reversible reaction, carry out CF-RM$_+$.
\item[6.] For a CF-subset with a reversible reaction, carry out CF-RM$_+$, but in addition, for each reversible reaction, also for the CF-subset of the reverse reaction (with the same ``catalytic'' complex). If the reactant complex of the reverse reaction is an NF-node, this additional step removes the original CF-subset from the reaction set of that NF-node. If this removal transforms the NF-node to a CF-node, then remove the node from the list of NF-nodes (to be processed).
\end{itemize}

\begin{remark}
It is in the last step that the resulting new reactant may be a non-multiple of the original reactant, since the ``catalytic'' complex added is determined by the reactant of the other reaction in the reversible pair.
\end{remark}

Two basic properties of the CF-RI$_+$ transformation are collected in the following proposition:
\begin{proposition}
Let $\mathscr{N}_{RI}$  be the CF-RI$_+$ transform of $\mathscr{N}$.
\begin{itemize}
\item[i.] If the CRN has no reversible reactions, then CF-RI$_+$ = CF-RM$_+$. 
\item[ii.] The stochiometric subspaces are equal, i.e., $S_{RI} =S$.
\end{itemize}
\end{proposition}

\section{The ${\mathscr{F}}$-decomposition under the CF-RI$_+$ Transformation}
\label{f:decomposition:cfri:trans}
The following theorem implies that with or without the application of the CF-RM transformation, the computation on determining whether a PL-NDK system has the capacity to admit multiple equilibria using the Multistationarity Algorithm are the same with the assumption of the independence of the $\mathscr{F}$-decomposition.

\begin{theorem}
Let $({\mathscr{N}}, K)$ be a PL-NDK system and $({\mathscr{N}}_{RI}, K_{RI})$ a CF-RI$_+$ transform. Then
\begin{itemize}
\item[i.] for any orientation ${\mathscr{O}}$ of ${\mathscr{N}}$, $|{\mathscr{O}}| = |{\mathscr{O}_{RI}}|$, and
\item[ii.] the ${\mathscr{F}}$-decomposition of ${\mathscr{N}}$ is independent if and only if the ${\mathscr{F}}$-decomposition of ${\mathscr{N}}_{RI}$ is independent.
\end{itemize}
\end{theorem}

\begin{proof}
Since the transformation preserves the reversibility and irreversibility of the reactions, $|{\mathscr{O}}| = |{\mathscr{O}_{RI}}|$.
Suppose the $\mathscr{F}$-decomposition of ${\mathscr{N}}$ is independent.
By Theorem \ref{PifC}, $\mathscr{P}$-decomposition of ${\mathscr{N}}$ is independent. Note that the reaction vectors remain the same after the application of the transformation. Moreover, we assume that the reversibility and irreversibility of the reactions are retained and we have
\[\displaystyle \sum\limits_{y \to y' \in \mathscr{O}} {{\alpha _{y \to y'}}} \left( {y' - y} \right) = 0 = \displaystyle \sum\limits_{y \to y' \in \mathscr{O}_{RI}} {{\alpha _{y \to y'}}} \left( {y' - y} \right).\]
Hence, we can choose the same basis for $Ker L_\mathscr{O}$ and $Ker L_{\mathscr{O}_{RI}}$ such that the order of the rows of the reactions corresponding to the basis remains the same. Thus, the equivalence classes are retained under the transformation. Therefore, the $\mathscr{P}$-decomposition of ${\mathscr{N}}_{RI}$ is independent. It follows that the $\mathscr{F}$-decomposition of ${\mathscr{N}}_{RI}$ is independent. The same proof applies for the converse.
\end{proof}

\begin{runningexample}
We apply the CF-RI$_+$ transform to the reaction network of the Schmitz's carbon cycle model. We modify $R_5$ and obtain the following dynamically equivalent PL-RDK which also has an independent $\mathscr{F}$-decomposition.
\[ \begin{array}{lll}
R_1: M_1 \to M_5 & \ \ \ &  R_5: 2M_1  \to M_1+M_3\\
R_2: M_5  \to  M_1  &  \ \ \ & R_6: M_3 \to M_4 \\
R_3: M_5  \to M_6 &  \ \ \ & R_{7}: M_4 \to M_2\\
R_4: M_6  \to M_1  &  \ \ \ & R_{8}: M_2 \to M_1  \\
\end{array}\]
\end{runningexample}

\section{Conclusion and Outlook}
\label{sec:conclusion}
We summarize our results and provide some direction for future research.

\begin{itemize}
\item[1.] We introduced the ${\mathscr{O}}$-, ${\mathscr{P}}$-, and ${\mathscr{F}}$-decompositions underlying the HDA which is the basis of the multistationarity algorithm MSA for power-law kinetics. We derive properties of these decompositions such as independence and incidence-independence, and identify network classes where the $\mathscr{F}$-decomposition coincides with other known decompositions.
\item[2.] We classified the $\mathscr{F}$-decomposition into three types according to the types of subnetworks induced by the decomposition. We explored the network properties of each of these types of decompositions.
\item[3.] As our major examples, we have determined that the CRNs of phosphorylation/ dephosphorylation systems have bi-independent $\mathscr{F}$-decompositions. We also used a subnetwork of the Schmitz's carbon cycle model as a running example. We have shown that the ${\mathscr{F}}$-decomposition is bi-independent. We generalized this type of subnetworks with a chain of long monomolecular directed cycles which can be possibly broken.
\item[4.] We have shown that for independent $\mathscr{F}$-decomposition, the additional CF-RM transformation is not needed, and hence the MSA can be applied directly to the system.
\item[5.] One can prove the results for a larger class containing the set of independent $\mathscr{F}$-decompositions.
\end{itemize}

\section*{Acknowledgement}
BSH acknowledges the support of DOST-SEI (Department of Science and Technology-Science
Education Institute), Philippines for the ASTHRDP Scholarship grant.

\appendix
\section{Nomenclature}
\label{nomenclature:appendix}
\subsection{List of abbreviations}
\begin{tabular}{ll}
\noalign{\smallskip}\hline\noalign{\smallskip}
Abbreviation& Meaning \\
\noalign{\smallskip}\hline\noalign{\smallskip}
CF& complex factorizable \\
CKS& chemical kinetic system\\
CRN& chemical reaction network\\
CRNT& Chemical Reaction Network Theory \\
GMA& generalized mass action\\
HDA& higher deficiency algorithm\\
MAK& mass action kinetics\\
MSA& multistationarity algorithm\\
PLK& power-law kinetics\\
PL-NDK& power-law non-reactant-determined kinetics\\
PL-RDK& power-law reactant-determined kinetics\\
SFRF& species formation rate function\\
\noalign{\smallskip}\hline
\end{tabular}
\subsection{List of important symbols}
\begin{tabular}{ll}
\noalign{\smallskip}\hline\noalign{\smallskip}
Meaning& Symbol \\
\noalign{\smallskip}\hline\noalign{\smallskip}
deficiency& $\delta$  \\
dimension of the stoichiometric subspace& $s$   \\
incidence map& $I_a$\\
molecularity matrix& $Y$\\
number of complexes& $n$\\
number of linkage classes& $l$\\
number of strong linkage classes& $sl$\\
orientation& $\mathscr{O}$\\
stoichiometric matrix& $N$\\
stoichiometric subspace& $S$\\
subnetwork of $\mathscr{N}$ with respect to $\mathscr{O}$& $\mathscr{N}_\mathscr{O}$\\
\noalign{\smallskip}\hline
\end{tabular}

\end{document}